\documentclass[12pt]{amsart}
\usepackage{amsmath,amsthm,amssymb, amsaddr}
\usepackage{setspace}
\usepackage[all]{xy} \xyoption{arc}
\usepackage[left=3cm,top=2cm,right=3cm,bottom = 2cm]{geometry}
\usepackage{graphicx}
\usepackage[usenames,dvipsnames]{color}
\usepackage{charter}
\usepackage{quoting}
\usepackage{hyperref}

\theoremstyle{plain}
\newtheorem{thm}{Theorem}[section]
\newtheorem{lem}[thm]{Lemma}
\newtheorem{prop}[thm]{Proposition}
\newtheorem{cor}[thm]{Corollary}

\theoremstyle{definition}
\newtheorem{dfn}[thm]{Definition}
\newtheorem{ex}[thm]{Example}

\theoremstyle{remark}
\newtheorem{rmk}[thm]{Remark}

\newcommand{\cF}{\mathcal{F}}

\newcommand{\veps}{\varepsilon}

\DeclareRobustCommand{\stirling}{\genfrac\{\}{0pt}{}}

\DeclareMathOperator{\Tr}{Tr}

\DeclareMathOperator{\Hom}{Hom}

\DeclareMathOperator{\Endo}{End}

\DeclareMathOperator{\im}{im}

\DeclareMathOperator{\SL}{SL}

\DeclareMathOperator{\Res}{Res}

\DeclareMathOperator{\id}{id}

\newcommand*{\df}{\mathrel{\vcenter{\baselineskip0.5ex \lineskiplimit0pt
                     \hbox{\scriptsize.}\hbox{\scriptsize.}}} =}

\providecommand{\abs}[1]{\left\lvert#1\right\rvert}

\providecommand{\pseries}[2]{#1[\![ #2 ]\!]}

\newcommand{\Qp}{\mathbf{Q}_p}
\newcommand{\Zp}{\mathbf{Z}_p}

\newcommand{\QQ}{\mathbf{Q}}

\newcommand{\CC}{\mathbf{C}}

\newcommand{\ZZ}{\mathbf{Z}}

\newcommand{\RR}{\mathbf{R}}

\newcommand{\bone}{\mathbf{1}}

\DeclareMathOperator{\kVer}{\mathbf{k-Ver}}

\numberwithin{equation}{section}

\begin{document}
\title{$p$-adic vertex operator algebras}
\author{Cameron Franc}
\address{Department of Mathematics and Statistics, McMaster University}
\email{franc@math.mcmaster.ca}
\author{Geoffrey Mason}
\address{Department of Mathematics, UC Santa Cruz}
\email{gem@ucsc.edu}
\thanks{The authors were supported by an NSERC grant and by the Simon Foundation, grant $\#427007$, respectively.\\
MSC(2020): 11F85, 17B69, 17B99, 81R10.}
\date{}

\begin{abstract} We postulate axioms for a chiral half of a nonarchimedean $2$-dimensional bosonic conformal field theory, that is, a vertex operator algebra
in which a $p$-adic Banach space replaces the traditional Hilbert space. We study some consequences of our axioms leading to the construction of various examples, including $p$-adic commutative Banach rings and $p$-adic versions of the Virasoro, Heisenberg, and the Moonshine module vertex operator algebras. Serre $p$-adic modular forms occur naturally in some of these examples as limits of classical $1$-point functions.
\end{abstract}
\maketitle
\tableofcontents

\section{Introduction}
In this paper we introduce a new notion of $p$-adic vertex operator algebra. The axioms of $p$-adic VOAs arise from $p$-adic completion of the usual or, as we call them in this paper, \emph{algebraic} VOAs with suitable integrality properties. There are a number of reasons why this direction is worth pursuing, both mathematical and physical. From a mathematical perspective, many natural and important VOAs have integral structures, including the Monster module of \cite{FLM}, whose integrality properties have been studied in a number of recent papers, including \cite{Carnahan, DongGriess1, DongGriess2, DongGriess3}. Completing such VOAs with respect to the supremum-norm of an integral basis as in Section \ref{s:completion} below provided the model for the axioms that we present here. In a related vein, the last several decades have focused attention on various aspects of algebraic VOAs over finite fields. Here we can cite \cite{Borcherds, BorcherdsRyba, DongRen1, DongRen2, LiMu1, LiMu2, Ryba}. It is natural to ask how such studies extend to the $p$-adic completion, and this paper provides a framework for addressing such questions.

In a slightly different direction, $p$-adic completions have been useful in a variety of mathematical fields  adjacent to the algebraic study of VOAs. Perhaps most impressive is the connection between VOAs and modular forms, discussed in \cite{DLMModular, FMCharacters, MT2, Zhu}. It is natural to ask to what extent this connection extends to the modern and enormous field of $p$-adic modular forms, whose study was initiated in \cite{Katz, Serre2} and which built on earlier work of many mathematicians. We provide some hints of such a connection in Sections \ref{s:heisenberg} and \ref{s:monster} below.

Finally on the mathematical side, local-global principles have been a part of the study of number-theoretic problems for over a century now, particularly in relation to questions about lattices and their genera.\ One can think of algebraic VOAs as enriched lattices (see \cite{FMCharacters} for some discussion and references on this point) and it is natural to ask to what extent local-global methods could be applied to the study of VOAs. For example, mathematicians still do not know how to prove that the Monster module is the unique holomorphic VOA with central charge $24$ and whose weight $1$ graded piece is trivial. Could the uniqueness of its $p$-adic completion be more accessible?

When considering how to apply $p$-adic methods to the study of vertex operator algebras, one is confronted with the question of which $p$ to use. In the algebraic and physical theories to date, most attention has been focused on the infinite, archimedean prime. A finite group theorist might answer that $p=2$ is a natural candidate. While focusing on $p=2$ could trouble a number theorist, there is support for this suggestion in recent papers such as \cite{Susskind} where the Bruhat-Tits tree for $\SL_2(\QQ_2)$ plays a prominent r\^{o}le. One might also speculate that for certain small primes, $p$-adic methods could shed light on  $p$-modular moonshine \cite{Borcherds, BorcherdsRyba, Ryba} and related constructions \cite{DHR}. However, a truly local-global philosophy would suggest considering all primes at once. That is, one might consider instead the adelic picture, an idea that has arisen previously in papers such as \cite{FreundWitten}. Examples of such objects are furnished by completions $\varprojlim_{n\geq 1} V/nV$ of vertex algebras $V$ defined over 
$\ZZ$. After extension by $\QQ$, such a completion breaks up into a restricted product over the various $p$-adic completions. Thus, the $p$-adic theory discussed below could be regarded as a first step toward a more comprehensive adelic theory. Other than these brief remarks, however, we say no more about the adeles in this paper.

There is a deep connection between the algebraic theory of VOAs and the study of physical fields connected to quantum and conformal field theory (CFT) which has been present in the algebraic theory of VOAs from its very inception. A `physical' $2$-dimensional CFT provides for a pair of Hilbert spaces of states, called the left- and right-moving Hilbert spaces, or  `chiral halves'. We will not need any details here; some of them are presented in \cite{Kac}, or for a more physical perspective one can consult \cite{FMS,Schottenloher}.

The current mathematical theory of vertex operator algebras is, with some exceptions, two steps removed from this physical picture: one deals exclusively with one of the chiral halves and beyond that the topology is rendered irrelevant by restricting to a dense subspace of states and the corresponding fields, which may be treated axiomatically and algebraically.

The $p$-adic theory we propose is but one step removed from the physical picture of left- and right-moving Hilbert spaces:  although we deal only with chiral objects, they are topologically complete. The metric is nonarchimedean and the Hilbert space is replaced by a $p$-adic Banach space. Thus our work amounts to an axiomatization of the chiral half of a nonarchimedean $2$-dimensional bosonic CFT.

 There is a long history of $p$-adic ideas arising in the study of string theory and related fields, cf. \cite{FreundOlson1, FreundWitten, Gubser, HuangStoicaYau, HungLiMelbyThompson, VVZ, Volovich} for a sample of some important works. In these papers the conformal fields of interest tend to be complex-valued as opposed to the $p$-adic valued fields discussed below. Nevertheless, the idea of a $p$-adic valued field in the sense meant by physicists has been broached at various times, for example in Section 4.4 of \cite{Marcolli}. 

The new $p$-adic axioms are similar to the usual algebraic ones, as most of the axioms of an algebraic VOA are stable under completion. For example, the formidable Jacobi identity axiom is unchanged, though it now involves infinite sums that need not truncate. On the other hand, in our $p$-adic story the notion of a \emph{field} in the sense of VOA theory must be adjusted slightly, and this results in a slightly weaker form of $p$-adic locality. Otherwise, much of the basic algebraic theory carries over to this new context. This is due to the fact that the strong triangle inequality for nonarchimedean metrics leads to analytic considerations that feel very much like they are part of algebra rather than analysis. Consequently, many standard arguments in the algebraic theory of VOAs can be adapted to this new setting.

The paper is organized as follows: to help make this work more accessible to a broader audience, we include in Sections \ref{s:primer} and \ref{s:banach} a recollection of some basic facts about algebraic VOAs and $p$-adic Banach spaces respectively. In Section \ref{s:padic1} we introduce our axioms for $p$-adic fields and $p$-adic vertex algebras and derive some basic results about them. In Section \ref{s:goddard} we establish $p$-adic variants of the Goddard axioms \cite{Goddard} as discussed  in \cite{Mason1, MatsuoNagatomo}. Roughly speaking, we show that a $p$-adic vertex algebra amounts to a collection of mutually $p$-adically local $p$-adic fields, a statement whose archimedean analog will be familiar to experts.\ In Section \ref{s:conformal} we discuss Virasoro structures, define $p$-adic vertex operator algebras, and constuct $p$-adic versions of the algebraic Virasoro algebra. Section \ref{s:completion} provides tools for constructing examples of $p$-adic VOAs via completion of algebraic VOAs. In Section \ref{s:locality} we elaborate on some aspects of $p$-adic locality and related topics such as operator product expansions. Finally, Sections \ref{s:heisenberg} and \ref{s:monster} use the preceding material and results in the literature to study examples of $p$-adic VOAs. In particular, we establish the following result:
\begin{thm}\label{thm1.1}
There exist $p$-adic versions of the Virasoro, Heisenberg and Monster VOAs.
\end{thm}
Perhaps as interesting 
 is the fact that in the second and third cases mentioned in Theorem \ref{thm1.1}, the character maps (i.e., $1$-point correlation functions, or graded traces)  for the VOAs extend by continuity to character maps giving rise to $p$-adic modular forms as defined in \cite{Serre2}. A noteworthy fact about the Heisenberg algebra is that while the quasi-modular form $E_2/\eta$ is the graded trace of a state in the Heisenberg algebra, if we ignore the factor of $\eta^{-1}$, then for odd primes $p$ this is a genuine $p$-adic modular form \`{a} la Serre. In this sense, the $p$-adic modular perspective may be \emph{more} attractive than the algebraic case where one must incorporate quasi-modular forms into the picture. See Sections \ref{s:heisenberg}, \ref{s:kummer} and \ref{s:monster} for more details on these examples. In particular, in Section \ref{s:kummer} we show, among other things, that the image of the $p$-adic character map for the $p$-adic Heisenberg algebra contains the $p$-stabilized Eisenstein series of weight $2$,
\[
  G_2^* \df \frac{p^2-1}{24}+\sum_{n\geq 1}\sigma^*(n)q^n,
\]
with notation as in \cite{Serre2}, so that $\sigma^*(n)$ denotes the divisor sum function over divisors of $n$ coprime to $p$. It is an interesting problem to determine the images of these $p$-adic character maps, a question we hope to return to in the near future.

In a similar way, one can deduce the existence of $p$-adic versions of many well-known VOAs, such as lattice theories, and theories modeled on representations of affine Lie algebras (WZW models). However, since it is more complicated to formulate definitive statements about the ($p$-adic) characters of such objects, and in the interest of keeping the discussion to a reasonable length, we confine our presentation to the cases intervening in Theorem \ref{thm1.1}. We hope that they provide a suitable demonstration and rationale for the theory developed below.

The authors thank  Jeff Harvey and Greg Moore for comments on a prior version of this paper.

\subsection{Notation and terminology}\label{SSterm}
Researchers in number theory and conformal field theory have adopted a number of conflicting  choices of terminology. In this paper \emph{field} can mean algebraic field as in the rational, complex or $p$-adic fields. Alternately, it can be a field as in the theory of VOAs or physics. Similarly, \emph{local} can refer to local fields such as the $p$-adic numbers, or it can refer to physical locality, as is manifest in the theory of VOAs. Throughout this paper, these words generally take their physically-inspired meaning as in the theory of VOAs.
\begin{enumerate}
\item[---] If $A$, $B$ are operators, then $[A,B] = AB-BA$ is their commutator;
\item[---] For a rational prime $p$, $\QQ_p$ and $\ZZ_p$ are the rings of $p$-adic numbers and $p$-adic integers respectively;
\item[---] If $V$ is a $p$-adic Banach space, then $\cF(V)$ is the space of $p$-adic fields on $V$, cf. Definition \ref{d:pfield};
\item[---] The Bernoulli numbers $B_k$ are defined by the power series
\begin{eqnarray*}
\frac{z}{e^z-1} = \sum_{k\geq 0} \frac{B_k}{k!} z^k;
\end{eqnarray*}
\end{enumerate}

\section{Primer on algebraic vertex algebras}
\label{s:primer}
In this Section we review the basic theory of \emph{algebraic} vertex algebras over an arbitrary unital, commutative base ring $k$, and we variously call such a gadget a \emph{$k$-vertex algebra}, \emph{vertex $k$-algebra}, or \emph{vertex algebra over $k$}.  Actually, our main interests reside in the cases when $k$ is a field of characteristic $0$, or $k=\ZZ$, $\ZZ/p^n\ZZ$, or the $p$-adic integers $\ZZ_p$, but there is no reason not to work in full generality, at least at the outset.\ By \emph{algebraic}, we mean the standard mathematical theory of vertex algebras based on the usual Jacobi identity (see below) as opposed to $p$-adic vertex algebras as discussed in the present paper.\ Good references for the theory over $\CC$ are \cite{FBZ, LepowskyLi} (see also \cite{MT2} for an expedited introduction) and the overwhelming majority of papers in the literature are similarly limited to this case.\ The literature on $k$-vertex algebras for other $k$, especially when $k$ is \emph{not} a $\QQ$-algebra, is scarce indeed. There is some work \cite{DongGriess1, DongGriess2, DongGriess3} on the case when $k=\ZZ$ that we shall find helpful --- see also \cite{Borcherds, BorcherdsRyba, Carnahan, DongRen1, DongRen2, JiaoLiMu, LiMu1, LiMu2, Ryba}. A general approach to $k$-vertex algebras is given in \cite{Mason1}.

\subsection{\texorpdfstring{$k$}{k}-vertex algebras}
\begin{dfn}\label{kvertdef} An (algebraic) $k$-vertex algebra is a triple $(V, Y, \bone)$ with the following ingredients:
\begin{itemize}
\item $V$ is a $k$-module, often referred to as Fock space;
\item $Y: V\rightarrow \pseries{\Endo_k(V)}{z, z^{-1}}$ is $k$-linear, and we write $Y(v, z)\df \sum_{n\in\ZZ} v(n)z^{-n-1}$ for the map $v \mapsto Y(v,z)$;
\item $\bone\in V$ is a distinguished element in $V$ called the vacuum element.
\end{itemize}
The following axioms must be satisfied for all $u, v, w\in V$:
\begin{enumerate}

\item (Truncation condition): there is an integer $n_0$, depending on $u$ and $v$, such that $u(n)v=0$ for all $n>n_0$;
\item (Creativity): $u(-1)\bone = u$, and $u(n)\bone = 0$ for all $n\geq 0$;
\item (Jacobi identity): the following identity is satisfied for all triples of integers $(r, s, t)$:
  \begin{align}
      \label{algJI}
&\sum_{i\geq 0} \binom{r}{i} (u(t+i)v)(r+s-i)w\\
    =&\sum_{i\geq 0}(-1)^i\binom{t}{i}\left( u(r+t-i)v(s+i)w-(-1)^t v(s+t-i)u(r+i)w\right).\notag
  \end{align}
\end{enumerate}
\end{dfn}

Inasmuch as $u(n)$ is a $k$-linear endomorphism of $V$ called the $n^{th}$ mode of $u$, we may think of $V$ as a $k$-algebra equipped with a countable infinity of $k$-bilinear products $u(n)v$. The \emph{vertex operator} (or field) $Y(u, z)$ assembles the modes of $u$ into a formal distribution and we can use an obvious notation $Y(u, z)v \df\sum_{n\in\ZZ} u(n)vz^{-n-1}$. Then the truncation condition says that $Y(u, z)v\in \pseries{V[z, z^{-1}]}{z}$ and the creativity axiom says that $Y(u, z)\bone=u+O(z)$. By deleting $w$ everywhere in the Jacobi identity, equation \eqref{algJI} may be construed as an identity satisfied by modes of  $u$ and $v$.

Some obvious but nevertheless important observations need to be made. For any given triple $(u, v, w)$ of elements in $V$, the truncation condition ensures that
\eqref{algJI} is well-defined in the sense that both sides reduce to \emph{finite} sums. Furthermore, only integer coefficients occur, so that \eqref{algJI} makes perfectly good sense for any commutative base ring $k$.

It can be shown as a consequence of these axioms \cite{Mason1} that $\bone$ is a sort of identity element, and more precisely that $Y(\bone, z)=\id_V$.

Suppose that $U=(U, Y, \bone), V=(V, Y, \bone)$ are two $k$-vertex algebras. A \emph{morphism}$f: U\rightarrow V$ is a morphism of $k$-modules that preserves vacuum elements and all $n^{th}$ products. This latter statement can be written in the form $fY(u, z)=Y(f(u),z)f$ for all $u\in U$. $k$-vertex algebras and their morphisms form a category $\kVer$. A \emph{left-ideal} in $V$ is a $k$-submodule $A\subseteq V$ such that $Y(u, z)a\in \pseries{A[z, z^{-1}]}{z}$ for all $u\in V$ and all $a\in A$. Similarly, $A$ is a right-ideal if $Y(a, z)u\in\pseries{A[z, z^{-1}]}{z}$. A $2$-sided ideal is, of course, a $k$-submodule that is both a left- and a right-ideal. If $A\subseteq V$ is a $2$-sided ideal then the quotient $k$-module $V/A$ carries the structure of a vertex $k$-algebra with the obvious vacuum element and $n^{th}$ products. Kernels of morphisms $f: U\rightarrow V$ are $2$-sided ideals and there is an isomorphism of $k$-vertex algebras $U/\ker f\cong \im f$.

\subsection{\texorpdfstring{$k$}{k}-vertex operator algebras}
A definition of $k$-vertex operator algebras for a general base ring $k$ is a bit complicated \cite{Mason1}, so we will limit ourselves to the standard case where $k$ is a $\QQ$-algebra. This will suffice for later purposes where we are mainly interested in considering $k=\Qp$, which is a field of characteristic zero.

Recall the Virasoro Lie algebra with generators $L(n)$ for $n\in\ZZ$ together with a central element $\kappa$ satisfying the bracket relations 
\begin{equation}
  \label{altVir}
[L(m), L(n)]=(m-n)L(m+n) + \delta_{m, -n} \tfrac{(m^3-m)}{12}\kappa.
\end{equation}

In a $k$-vertex operator algebra the Virasoro relations \eqref{altVir} are blended into a $k$-vertex algebra as follows:
\begin{dfn}
  \label{kvertopalgdef} A $k$-vertex operator algebra is a quadruple $(V, Y, \bone, \omega)$ with the following ingredients:
\begin{itemize}
\item A $k$-vertex algebra $(V, Y, \bone)$;
\item $\omega\in V$ is a distinguished element called the Virasoro element, or conformal vector;
\end{itemize}
and the following axioms are satisfied:
\begin{enumerate}
\item $Y(\omega, z)= \sum_{n\in\ZZ} L(n)z^{-n-2}$ where the modes $L(n)$ satisfy the relations \eqref{altVir} with $\kappa=c\id_V$ for some scalar $c\in k$ called the central charge of $V$;
\item there is a direct sum decomposition $V=\oplus_{n\in\ZZ} V_n$ where
  \[V_n\df \{v\in V\mid L(0)v=nv\}\]
  is a finitely generated $k$-module, and $V_n=0$ for $n\ll 0$;
\item $[L(-1), Y(v, z)]=\partial_zY(v, z)$ for all $v\in V$.
\end{enumerate}
\end{dfn}

This definition deserves a lot of explanatory comment - much more than we will provide. That the modes $L(n)$ satisfy the Virasoro relations means, in effect, that $V$ is a module over the Virasoro algebra. Of course it is a very special module, as one sees from the last two axioms. Furthermore, because $(V, Y, \bone)$ is a $k$-vertex algebra, then the Jacobi identity \eqref{algJI} must be satisfied by the modes of all elements $u, v \in V$, including for example $u=v=\omega$. For $\QQ$-algebras $k$ this is a non-obvious but well-known fact. Indeed, it holds for any $k$, cf. \cite{Mason1}, but this requires more discussion than we want to present here, this being one of the reasons we limit our choice of base ring.

Despite their relative complexity, there are large swaths of $k$-vertex operator algebras, especially in the case when $k=\CC$ is the field of complex numbers. See the references above for examples.
\subsection{Lattices in \texorpdfstring{$k$}{k}-vertex operator algebras}
Our main source of examples of $p$-adic VOAs will be completions of algebraic VOAs with suitable integrality properties, codified in the existence of integral forms as defined below. This definition is modeled on \cite{DongGriess1}, but see also \cite{Carnahan} for a more general perspective on vertex algebras over schemes.

Let either $k = \QQ$ or $\QQ_p$ and let $A=\ZZ$ or $\ZZ_p$, respectively.
\begin{dfn}
  \label{d:lattice}
  An \emph{integral form} in a $k$-vertex operator algebra $V$ is an $A$-submodule $R\subseteq V$ such that
  \begin{enumerate}
\item[(i)] $R$ is an $A$-vertex subalgebra of $V$, in particular $\bone \in R$,
\item[(ii)] $R_{(n)}\df R \cap V_{(n)}$ is an $A$-base of $V_{(n)}$ for each $n$,
\item[(iii)] there is a positive integer $s$ such that $s\omega\in R$.
\end{enumerate}
\end{dfn}
Condition (i) above means in particular that if $v \in R$, then every mode $v(n)$ defines an $A$-linear endomorphism of $R$. In particular, if $R$ is endowed with the sup-norm in some graded basis, and if this is extended to $V$, then the resulting modes are uniformly bounded in the resulting $p$-adic topology, as required by the definition of a $p$-adic field, cf. part (1) of Definition \ref{d:pfield} below.

\subsection{Remarks on inverse limits}\label{SSprojlim}
Suppose now that we consider a $\ZZ_p$-vertex algebra $V$. For a positive integer $k$, $p^kV$ is a $2$-sided ideal in $V$, so that the quotient $V/p^kV$ is again a $\ZZ_p$-vertex algebra. (We could also consider this as a vertex algebra over $\ZZ/p^k\ZZ$.) We have canonical surjections of $\ZZ_p$-vertex rings $f^m_k: V/p^{m}V \rightarrow V/p^{k}V$ for $m\geq k$ and we may consider the inverse limit
\begin{equation}
\varprojlim V/p^kV.
\end{equation}

This is unproblematic at the level of $\ZZ_p$-modules. Elements of the inverse limit are sequences $(v_1, v_2, v_3, \ldots)$ such that $f^m_k(v_m)=v_k$ for all $m\geq k$. It is natural to define the $n^{th}$ mode of such a sequence to be
$(v_1, v_2, v_3, \ldots)(n)\df (v_1(n), v_2(n), v_3(n), \ldots)$. However, this device does not make the inverse limit into an algebraic vertex algebra over $\ZZ_p$  in general.\ This is because it is not possible, in the algebraic setting, to prove the truncation condition for such modal sequences --- see Section \ref{s:heisenberg} below for a concrete example.\ And without the truncation condition, the Jacobi identity becomes meaningless even though in some sense it holds formally.\  As we shall explain below, these problems can be overcome in a suitable $p$-adic setting by making use of the $p$-adic topology. Indeed, this observation provides the impetus for many of the definitions to follow.

\section{\texorpdfstring{$p$}{p}-adic Banach spaces}
\label{s:banach}
A \emph{Banach space} $V$ over $\Qp$, or $p$-adic Banach space, is a complete normed vector space over $\Qp$ whose norm satisfies the ultrametric inequality
\[
  \abs{x+y} \leq \sup(\abs{x},\abs{y})
\]
for all $x,y \in V$. See the encyclopedic \cite{BGR} for general background on nonarchimedean analysis with a view towards rigid geometry.  Following Serre \cite{Serre}, we shall assume that for every $v\in V$ we have $\abs{v} \in \abs{\Qp}$, so that under the standard normalization for the $p$-adic absolute value, we can write $\abs{v} = p^n$ for some $n\in\ZZ$ if $v\neq 0$. Note that we drop the subscript $p$ from the valuation notation to avoid a proliferation of such subscripts. Since $\Qp$ is discretely valued, this is the same as Serre's condition (N). We will often omit mention of this condition throughout the rest of this note. It is likely inessential and could be removed, for example if one wished to consider more general nonarchimedean fields.

A basic and well-known consequence of the ultrametric inequality in $V$ that we will use repeatedly is the following:
\begin{lem}
  \label{l:sums}
Let $v_n \in V$ be a sequence in a $p$-adic Banach space $V$. Then $\sum_{n=1}^\infty v_n$ converges if and only if $\lim_{n \to \infty} v_n=0$. $\hfill\Box$
\end{lem}

Proposition 1 of \cite{Serre} and the ensuing discussion shows that, up to continuous isomorphism, every $p$-adic Banach space can be described concretely as follows. Let $I$ be a set and define $c(I)$ to be the collection of families $(x_i)_{i\in I} \in \Qp^I$ such that $x_i$ tends to zero in the following sense: for every $\veps > 0$, there is a finite set $S\subseteq I$ such that $\abs{x_i} < \veps$ for all $i \in I\setminus S$.\ Then $c(I)$ has a well-defined supremum norm $\abs{x} = \sup_{i\in I}\abs{x_i}$.\ Notice that since the $p$-adic absolute value is discretely valued, this norm takes values in the same value group as $\Qp$.\ Serre shows that every $p$-adic Banach space is continuously isomorphic to a $p$-adic Banach space of the form $c(I)$.
\begin{dfn}
  \label{d:orthonormal}
  Let $V$ be a $p$-adic Banach space. Then a sequence $(e_i)_{i\in I} \in \Qp^I$ is said to be an \emph{orthonormal basis} for $V$ provided that every element $x \in V$ can be expressed uniquely as a sum $x = \sum_{i \in I} x_ie_i$ for $x_i \in \Qp$ tending to $0$ with $i$, such that $\abs{x} = \sup_{i\in I} \abs{x_i}$.
\end{dfn}
That every $p$-adic Banach space is of the form $c(I)$, up to isomorphism, is tantamount to the existence of orthonormal bases.

The underlying Fock spaces of  vertex algebras over $\QQ_p$ will be $p$-adic Banach spaces. To generalize the definitions of $1$-point functions to the $p$-adic setting, it will be desirable to work in the context of trace-class operators. We thus recall some facts on linear operators between $p$-adic Banach spaces.
\begin{dfn}
\label{d:continuous}
  If $U$ and $V$ are $p$-adic Banach spaces, then $\Hom(U,V)$ denotes the set of \emph{continuous} $\Qp$-linear maps $U \to V$. If $U=V$ then we write $\Endo(V) = \Hom(V,V)$.
\end{dfn}

The space $\Hom(U,V)$ is endowed with the usual \emph{sup-norm}
\[\abs{f} = \sup_{u \neq 0} \frac{\abs{f(u)}}{\abs{u}}.\]
Thanks to our hypotheses on the norm groups of $U$ and $V$ it follows that $\abs{f} = \sup_{\abs{x}\leq 1}\abs{f(x)}$. This norm furnishes $\Hom(U,V)$ with the structure of a $p$-adic Banach space.

\section{\texorpdfstring{$p$}{p}-adic fields and \texorpdfstring{$p$}{p}-adic vertex algebras}
\label{s:padic1}
\begin{dfn}
  \label{d:pfield}
  Let $V$ be a $p$-adic Banach space. A \emph{$p$-adic field} on $V$ associated to a state $a \in V$ consists of a series $a(z) \in \pseries{\Endo(V)}{z,z^{-1}}$ such that if we write $a(z) = \sum_{n\in \ZZ} a(n)z^{-n-1}$ then:
  \begin{enumerate}
  \item there exists $M\in\RR_{\geq 0}$ depending on $a$ such that $\abs{a(n)b}\leq M\abs{a}\abs{b}$ for all $n\in \ZZ$ and all $b \in V$;
  \item $\lim_{n\to \infty} a(n)b=0$ for all $b \in V$.
  \end{enumerate}
\end{dfn}

\begin{rmk}
Property (1) in Definition \ref{d:pfield} implies that the operators $a(n) \in \Endo(V)$ are uniformly $p$-adically bounded $\abs{a(n)} \leq M\abs{a}$ in their operator norms, defined above in Section \ref{s:banach} and recalled below. In particular, the modes $a(n)$ are continuous endomorphisms of $V$ for all $n$.\ Property (2) of Definition \ref{d:pfield} arises by taking limits of the truncation condition
(1) in Definition \ref{kvertdef}.
\end{rmk}

We single out a special case of Definition \ref{d:pfield}, of particular interest not only because of its connection with Banach rings (see below) but also because it is satisfied in most, if not all, of our examples.
\begin{dfn}
  \label{d:pfieldsn}
  Let $V$ be a $p$-adic Banach space. A $p$-adic field $a(z)$ on $V$ is called \emph{submultiplicative} provided that it satisfies the following qualitatively stronger version of (1):
    \begin{enumerate}
  \item[(1')]   $\abs{a(n)b}\leq \abs{a}\abs{b}$ for all $n\in \ZZ$ and all $a, b \in V$.
  \end{enumerate}
\end{dfn}

\begin{rmk}
In the theory of nonarchimedean Banach spaces, where each $a \in V$ defines a single multiplication by $a$ operator rather than an entire sequence of such operators, the analogues of Property (1) in both Definition \ref{d:pfield} and Definition \ref{d:pfieldsn} above are shown to be equivalent in a certain sense: the constant $M$ of Definition \ref{d:pfield} can be eliminated at a cost of changing the norm, but without changing the topology. See Proposition 2 of Section 1.2.1 in  \cite{BGR} for a precise statement. We suspect that the same situation pertains in the theory of $p$-adic vertex algebras.\ Most of the arguments below are independent of this choice of definition, and so we work mostly with the apparently weaker Definition \ref{d:pfield}.
\end{rmk}

In the following, when we refer to ($p$-adic) fields, we will always mean in the sense of Definition \ref{d:pfield}. If the intention is to refer to submultiplicative fields we will invariably say so explicitly. Notice that if $a(z)$ is a field, then we can define a norm
\[
  \abs{a(z)} = \sup_{n\in\ZZ} \abs{a(n)},
\]
where as usual $\abs{a(n)}$ denotes the operator norm:
\[
  \abs{a(n)} \df \sup_{\substack{b\in V\\ b\neq 0}} \frac{\abs{a(n)b}}{\abs{b}}.
\]

\begin{dfn}\label{dfnpadicfd}
  The space of $p$-adic fields on $V$ in the sense of Definition \ref{d:pfield}, endowed with the topology arising from the sup-norm defined above, is denoted $\cF(V)$, and $\cF_{s}(V)$ denotes the subset of submultiplicative fields.
\end{dfn}

\begin{prop}\label{propcFV}
 $\cF(V)$ is a Banach space over $\Qp$. 
\end{prop}
\begin{proof}
Clearly $\cF(V)$ is closed under rescaling. To show that it is closed under addition, let $a(z)$, $b(z) \in \cF(V)$ be $p$-adic fields.  Property (2) of Definition \ref{d:pfield} is clearly satisfied by the sum $a(z)+b(z)$. We must show that Property (1) holds as well, so take $c \in V$. Let $M_1$ and $M_2$ be the constants of Property (1) arising from $a(z)$ and $b(z)$, respectively. Then we are interested in bounding
  \[
  \abs{a(n)c+b(n)c}\leq \max(\abs{a(n)c},\abs{b(n)c}) \leq \max(M_1,M_2)\abs{c}.
\]
Thus Property (1) of Definition \ref{d:pfield} holds for $a(z)+b(z)$ with $M= \max(M_1,M_2)$. This verifies that $\cF(V)$ is a subspace of $\pseries{\Endo(V)}{z,z^{-1}}$.

It remains to prove that $\cF(V)$ is complete. Let $a_j(z)$ be a Cauchy sequence in $\cF(V)$. This means that for all $\veps > 0$, there exists $N$ such that for all $i,j > N$,
\[
  \sup_{n \in \ZZ} \abs{a_i(n)-a_j(n)} =\abs{a_i(z)-a_j(z)} < \veps.
\]
In particular, for each $n\in\ZZ$, the sequence $(a_j(n))_{j\geq 0}$ of elements of $\Endo(V)$ is Cauchy with a well-defined limit $a(n) \df \lim_{j\to \infty}a_j(n)$.

We shall show that $a(z) = \sum_{n\in\ZZ} a(n)z^{-n-1}$ is a $p$-adic field. Let $\veps > 0$ be given. By definition of the sup-norm on $p$-adic fields, there exists $N$ such that for all $j > N$ and all $n\in\ZZ$,
\[
  \abs{a(n)-a_j(n)} < \veps.
\]
Let $b \in V$. Then for any choice of $j > N$,
\[
  \abs{a(n)b} \leq \sup(\abs{a(n)b-a_j(n)b},\abs{a_j(n)b})\leq \sup(\veps,M_j)\abs{b}
\]
where $M_j$ is the constant of Property (1) in Definition \ref{d:pfield} associated to $a_j(z)$. It follows that $a(z)$ satisfies Property (1) of Definition \ref{d:pfield} with $M = \sup(\veps,M_j)$, which is indeed independent of $b \in V$.

Let $b \in V$ be nonzero and fixed. To show that $\lim_{n\to\infty} a(n)b=0$, let $\veps > 0$ be given, and choose $j$ such that $\abs{a(n)-a_j(n)} < \frac{\veps}{\abs{b}}$ for all $n \in \ZZ$. That is,
\[
  \sup_{b' \neq 0} \frac{\abs{a(n)b'-a_j(n)b'}}{\abs{b'}} < \frac{\veps}{\abs{b}}.
\]
Then as above we have
\[
  \abs{a(n)b} \leq \sup(\abs{a(n)b-a_j(n)b},\abs{a_j(n)b}) < \sup(\veps,\abs{a_j(n)b}).
\]
Since $a_j$ is a $p$-adic field, there exists $N$ such that for $n > N$, we have $\abs{a_j(n)b} < \veps$. We thus see that for $n > N$, we have $\abs{a(n)b} < \veps$. Therefore $\lim_{n\to \infty} a(n)b = 0$. This verifies that the limit $a(z)$ of the $p$-adic fields $a_j(z)$ is itself a $p$-adic field, and therefore $\cF(V)$ is complete.
\end{proof}

\begin{rmk}\label{rmkFs}  The set of fields $\cF_s(V)$ is not necessarily a linear subspace of $\cF(V)$, so there is no question of an analog of Proposition \ref{propcFV} for submultiplicative fields.
\end{rmk}
Now suppose that we have a continuous $\QQ_p$-linear map
\[
  Y(\bullet,z) \colon V \to \cF(V).
\]
In the submultiplicative case the continuity hypothesis is automatically satisfied, as in the following result.
\begin{lem}
  \label{l:Ycont}
  Let $Y(\bullet,z) \colon V \to \cF(V)$ be a linear map associating a \emph{submultiplicative} $p$-adic field $Y(a,z)$ to each state $a \in V$. Then $Y$ is necessarily continuous.
\end{lem}
\begin{proof}
  By Proposition 2 in Section 2.1.8 of \cite{BGR}, in this setting, continuity and boundedness are equivalent, though we only require the easier direction of this equivalence. We have
  \[
\abs{Y(a,z)} = \sup_{n} \abs{a(n)} =\sup_{n} \sup_{\abs{b}=1} \abs{a(n)b}\leq \sup_{n} \sup_{\abs{b}=1} \abs{a}\abs{b}= \abs{a}.
\]
In particular, $Y(\bullet,z)$ is a bounded map and thus continuous.
\end{proof}

For every $u,v \in V$, we have $u(n)v\in V$ and so there are well-defined modes $(u(n)v)(m)$ arising from the $p$-adic field $Y(u(n)v,z)$. The following Lemma concerning these modes will allow us to work with the usual Jacobi identity from VOA theory in this new $p$-adic context.

\begin{lem}\label{lemJI}
 Let $u,v,w\in V$. Then for all $r,s,t \in \ZZ$, the infinite sum
   \begin{align*}
     &\sum_{i= 0}^{\infty} \binom{r}{i} (u(t+i)v)(r+s-i)w -\\
     &\sum_{i= 0}^{\infty} (-1)^i\binom{t}{i} \left\{u(r+t-i)(v(s+i)w)-(-1)^tv(s+t-i)(u(r+i)w) \right\}
  \end{align*}
  converges in $V$.
\end{lem}
\begin{proof}
   Since $\abs{n} \leq 1$ for all $n \in \ZZ$ by the strong triangle inequality, Lemma \ref{l:sums} implies that it will suffice to show that
  \begin{align*}
    \lim_{i\to \infty}\abs{(u(t+i)v)(r+s-i)w} &=0,\\
    \lim_{i \to \infty}\abs{u(r+t-i)(v(s+i)w)} &=0,\\
    \lim_{i\to \infty}\abs{v(s+t-i)(u(r+i)w)} &=0.
  \end{align*}
  
  We first discuss the second limit above, so let $M$ be the constant from Property (1) of Definition \ref{d:pfield} applied to $u$. Then we know that
  \[
\abs{u(r+t-i)(v(s+i)w)}\leq M\abs{v(s+i)w}.
\]
By Property (2) of Definition \ref{d:pfield}, this goes to zero as $i$ tends to infinity. This establishes the vanishing of the second limit above, and the third is handled similarly.

The first limit requires slightly more work. Notice that if $w =0$ there is nothing to show, so we may assume $w \neq 0$. Since $\lim_{i \to \infty} u(t+i)v =0$, continuity of $Y(\bullet,z)$ implies that
\[
\lim_{i \to \infty} Y(u(t+i)v,z) = 0.
\]
Hence, given $\veps > 0$, we can find $N$ such that for all $i > N$, we have $\abs{Y(u(t+i)v,z)} < \frac{\veps}{\abs{w}}$. By definition of the norm on $p$-adic fields, this means that
\[
  \sup_n\abs{(u(t+i)v)(n)w} < \veps.
\]
In particular, taking $n=r+s-i$, then for $i > N$ we have $\abs{(u(t+i)v)(r+s-i)w} < \veps$ as required. This concludes the proof.
\end{proof}

The reader should compare the next Definition with Definition \ref{kvertdef}.
 \begin{dfn}
  \label{d:pva}
  A \emph{$p$-adic vertex algebra} is a triple $(V, Y, \bone)$ consisting of a $p$-adic Banach space $V$ equipped with a distinguished state (the vacuum vector) $\bone \in V$ and a $p$-adic vertex operator, that is, a continuous $p$-adic linear map
  \[Y\colon V \to \cF(V),\]
  written $Y(a,z) = \sum_{n \in \ZZ} a(n)z^{-n-1}$ for $a\in V$,  satisfying the following conditions:
  \begin{enumerate}
  \item (Vacuum normalization) $\abs{\bone}\leq 1$.
  \item (Creativity) We have $Y(a,z)\bone \in a + z\pseries{V}{z}$, in other words,  $a(n)\bone=0$ for $n\geq 0$ and $a(-1)\bone=a$.
    \item (Jacobi identity) Fix  any $r,s,t \in \ZZ$ and $u,v,w \in V$. Then
  \begin{align*}
&\sum_{i= 0}^{\infty} \binom{r}{i} (u(t+i)v)(r+s-i)w =\\
 &\sum_{i= 0}^{\infty} (-1)^i\binom{t}{i} \left\{u(r+t-i)v(s+i)w-(-1)^tv(s+t-i)u(r+i)w \right\}.
  \end{align*}
\end{enumerate}
\end{dfn}

\begin{dfn}
  \label{d:pvasm}
  A $p$-adic vertex algebra $V$ is said to be \emph{submultiplicative} provided that every $p$-adic field $Y(a,z)$ is submultiplicative.
\end{dfn}
\begin{rmk}
If $V$ is a submultiplicative $p$-adic vertex algebra, then Lemma \ref{l:Ycont} implies that continuity of the state-field correspondence follows from the other axioms.
\end{rmk}

As in the usual algebraic theory of Section \ref{s:primer}, the vertex operator $Y$ is sometimes also called the \emph{state-field correspondence}, and both it and the vacuum vector will often be omitted from the notation. That is, we shall often simply say that $V$ is a $p$-adic vertex algebra. The vacuum property $Y(\bone, z)=\id_V$ follows from these axioms and does not need to be itemized separately. See Theorem \ref{thmvac} below for more details on this point.

\begin{rmk}\label{JIrmk}
  In the formulation of the Jacobi identity we have used the completeness of the $p$-adic Banach space $V$ to ensure that the infinite sum exists, via Lemma \ref{lemJI}. If $V$ is not assumed complete, one could replace the Jacobi identity  with corresponding \emph{Jacobi congruences}: fix $r,s,t\in \ZZ$ and $u,v,w\in\ZZ$. Then the Jacobi congruences insist that for all $\veps > 0$, there exists $j_0$ such that for all $j \geq j_0$, one has
  \begin{align*}
&\left\lvert\sum_{i= 0}^{j} \binom{r}{i} (u(t+i)v)(r+s-i)w -\right.\\
 &\left.\sum_{i= 0}^{j} (-1)^i\binom{t}{i} \left\{u(r+t-i)(v(s+i)w)-(-1)^tv(s+t-i)(u(r+i)w) \right\}\right\rvert < \veps.
  \end{align*}
  In the presence of completeness, this axiom is equivalent to the Jacobi identity.
\end{rmk}

\subsection{Properties of $p$-adic vertex algebras}
In this Subsection we explore some initial consequences of Definition \ref{d:pva}.
\begin{prop}
  \label{p:closedmap}
If $V$ is a $p$-adic vertex algebra, then for $a \in V$ we have
\[
\abs{a} \leq \abs{Y(a,z)}.
\]
In particular, the state-field correspondence $Y$ is an injective closed mapping. If $V$ is furthermore assumed to be submultiplicative, then $\abs{a} = \abs{Y(a,z)}$ for all $a \in V$.
\end{prop}
\begin{proof}
  If $a \in V$ then $a(-1)\bone = a$ by the creativity axiom, and so
  \[\abs{a}=\abs{a(-1)\bone}\leq \abs{Y(a, z)}\abs{\bone} \leq \abs{Y(a,z)},\]
  where the last inequality follows by vacuum normalization. This proves the first claim, and the second follows easily from this.

  For the final claim, observe that the proof of Lemma \ref{l:Ycont} shows that when $V$ is submultiplicative, we also have $\abs{Y(a,z)} \leq \abs{a}$. This concludes the proof.
\end{proof}

 \begin{dfn}
   \label{d:integral}
  A $p$-adic vertex algebra $V$ is said to \emph{have an integral structure} provided that for the $\ZZ_p$-module $V_0 \df \{v\in V \mid \abs{v} \leq 1\}$, the following condition holds:
  \begin{enumerate}
   \item the restriction of the state-field correspondence $Y$ to $V_0$ defines a map
    \[
  Y \colon V_0 \to \pseries{\Endo_{\Zp}(V_0)}{z,z^{-1}}.
    \]
  \end{enumerate}
\end{dfn}

Note that $\bone\in V_0$ by definition, and if $a\in V_0$ then $Y(a, z)\bone \in a+z\pseries{V_0}{z}$. In effect, then, the triple $(V_0, \bone, \Res_{V_0}Y)$ is a $p$-adic vertex algebra over $\ZZ_p$, though we have not formally defined such an object.

The axioms for a $p$-adic vertex algebra are chosen so that the following holds:
\begin{lem}
  \label{l:reduction}
Suppose that $V$ is a $p$-adic vertex algebra with an integral structure $V_0$.\ Then $V_0/p^kV_0$ inherits a natural structure of algebraic vertex algebra over $\ZZ/p^k\ZZ$ for all $k\geq 0$.
\end{lem}
\begin{proof}
  For a formal definition of a vertex ring over $\ZZ/p^k\ZZ$, the reader can consult \cite{Mason1}. Since we do not wish to go into too many details on this point, let us simply point out that if $a,b\in V_0$, then since $a(n)b$ tends to $0$ as $n$ grows, it follows that for any $k\geq 1$, the reduced series
  \[
    \sum_{n\in\ZZ} a(n)b z^{-n-1} \pmod{ p^kV_0} \in \pseries{\Endo_{\ZZ/p^k\ZZ}(V_0/p^kV_0)}{z,z^{-1}}
  \]
  has a finite Laurent tail. Since this is the key difference between algebraic and $p$-adic vertex rings, one deduces the Lemma from this fact.
\end{proof}

\section{\texorpdfstring{$p$}{p}-adic Goddard axioms}
\label{s:goddard}
In this Section we draw some standard consequences from the $p$-adic Jacobi identity in \eqref{d:pva}. The general idea is to show that the axioms for a $p$-adic vertex algebra, especially the Jacobi identity,  are equivalent to an alternate set of axioms that are ostensibly more intuitive and easier to recognize and manipulate.\ In the classical case of algebraic vertex algebras over $\CC$, or indeed any commutative base ring $k$, these are known as Goddard axioms \cite{Goddard, Mason1, MatsuoNagatomo}. At the same time, we develop some facts that we use later. Let $V$ be a $p$-adic vertex algebra.
\subsection{Commutator, associator and locality formulas}\label{SSacloc}
We begin with some immediate consequences of the Jacobi identity.
\begin{prop}[Commutator formula]
  \label{p:commutator}
For all $r, s\in\ZZ$ and all $u, v, w\in V$ we have
\begin{equation}
    \label{commform}
[u(r), v(s)]w= \sum_{i=0}^{\infty}\binom{r}{i} (u(i)v)(r+s-i)w.
\end{equation}
\end{prop}
\begin{proof}
Take $t=0$ in the Jacobi identity.
\end{proof}

\begin{prop}[Associator formula]\label{p:associator}
For all $s, t\in\ZZ$ and all $u, v, w\in V$ we have
\begin{equation}
\label{assocform}
(u(t)v)(s)w= \sum_{i=0}^{\infty} (-1)^i\binom{t}{i} \left\{u(t-i)v(s+i)w-(-1)^tv(s+t-i)u(i)w \right\}. 
\end{equation}
\end{prop}
\begin{proof}
Take $r=0$ in the Jacobi identity.
\end{proof}

The next result requires slightly more work.
\begin{prop}[$p$-adic locality]
  Let  $u, v, w \in V$. Then
    \begin{equation}\label{locform}
      \lim_{t\to \infty} (x-y)^t[Y(u, x), Y(v, y)]w = 0.
    \end{equation}
  \end{prop}
\begin{proof}
Because $\lim_{n \to \infty} a(n)b=0$, continuity of $Y$ yields $\lim_{n\to \infty} Y(a(n)b,z) =0$ in the uniform sup-norm. Thus, as in the last stages of the proof of Lemma \ref{lemJI}, for any $\veps>0$ and fixed $r$, $s$, $i$, we have the inequality $\abs{(u(t+i)v)(r+s-i)w}<\veps$ for all large enough $t$. 
Apply this observation to the left side of the Jacobi identity to deduce the following:

Let $r, s \in \ZZ$ and $u, v, w \in V$. For any $\veps>0$ there is an integer $N$ such that for all $t>N$ we have
\[
\abs{\sum_{i=0}^{\infty} (-1)^i\binom{t}{i}\left( u(r+t-i)v(s+i)w-(-1)^t v(s+t-i) u(r+i)w\right)} < \veps.\]

By a direct calculation, we can see that the summation on the left side of this inequality is exactly the coefficient of $x^ry^s$ in $(x-y)^t[Y(u, x), Y(v, y)]w$. Thus, we arrive at the statement of $p$-adic locality, and this concludes the proof.
\end{proof}

\begin{dfn}
  \label{d:locality}
  If equation \eqref{locform} holds for two $p$-adic fields $Y(u,x)$ and $Y(v,y)$, we say that $Y(u, x)$ and $Y(v, y)$ are \emph{mutually local $p$-adic fields}. When the context is clear, we will drop the word $p$-adic from the language.
\end{dfn}
\subsection{Vacuum vector}\label{SSvac}
We will prove
\begin{thm}
  \label{thmvac} Suppose that $(V, Y, \bone)$ is a p-adic vertex algebra. Then
\[
Y(\bone, z) = \id_V.
\]
\end{thm}
\begin{proof}
  We have to show that for all states $u\in V$ we have 
\begin{equation}\label{eq0}
\bone(n)u = \delta_{n, -1}u.
\end{equation}
 To this end, first we assert that for all $s\in\ZZ$,
\begin{equation}\label{eq1}
u(-1)\bone(s)\bone = \bone(s)u.
\end{equation}
This follows directly from the creativity axiom and the special case of the commutator formula \eqref{commform} in which $v=w=\bone$ and $r=-1$, $t=0$.

Now, by the creativity axiom once again, we have $\bone(s)\bone=0$ for $s\geq 0$ as well as $\bone(-1)\bone=\bone$. Feed these inequalities into \eqref{eq1} to see that \eqref{eq0} holds whenever $n\geq -1$.

Now we prove \eqref{eq0} for $n\leq-2$ by downward induction on $n$.\ First we choose $u=v=w=\bone$ and $t=-1$ in the Jacobi identity and fix $r, s\in\ZZ$, to see that
\begin{align*}
&\sum_{i= 0}^{\infty} \binom{r}{i} (\bone(-1+i)\bone)(r+s-i)\bone \\
=& \sum_{i= 0}^{\infty} \left\{\bone(r-1-i)(\bone(s+i)\bone)+\bone(s-1-i)(\bone(r+i)\bone) \right\} ,
\end{align*}
and therefore since we already know that $\bone(n)\bone=\delta_{n, -1}\bone$ for $n\geq -1$ we obtain
  \begin{equation}\label{eq2}
\bone(r+s)\bone =
  \sum_{i= 0}^{\infty} \left\{\bone(r-1-i)(\bone(s+i)\bone)+\bone(s-1-i)(\bone(r+i)\bone) \right\} 
  \end{equation}
  Specialize \eqref{eq2} to the case $r=s=-1$ to see that
   \[
 \bone(-2)\bone =
  \left\{\bone(-2)(\bone(-1)\bone)+\bone(-2)(\bone(-1)\bone) \right\}=2\bone(-2)\bone,
  \]
  whence $\bone(-2)\bone=0$. This begins the induction.
 
  Now choose $r+s=-n$ with $0\leq r<-s-1$. Then \eqref{eq2} reads
   \begin{equation}\label{eq3}
\bone(-n)\bone =
  \sum_{i= 0}^{\infty} \bone(r-1-i)\bone(s+i)\bone 
  \end{equation}
  By induction, the expression in \eqref{eq3} under the summation sign vanishes whenever $s+i>-n$ and $s+i\neq-1$, i.e., $i>r$, $i\neq-s-1$. It also vanishes if $i< r$ thanks to the commutator formula \eqref{commform}. So, the only possible nonzero contribution arises when $i=r$ and $i=-s-1$, in which case we obtain $\bone(-n)\bone = 2\bone(-n)\bone$ and therefore $\bone(-n)\bone=0$.
\end{proof}

\subsection{Translation-covariance and the canonical derivation $T$}\label{SStranscov} The canonical endomorphism $T$ of $V$ is defined by the formula
\begin{align}\label{Tdef}
Y(a, z)\bone &= a+T(a)z+ O(z^2),
\end{align}
that is, $T(a) \df a(-2)\bone$. The endomorphism $T$ is called the \emph{canonical derivation} of $V$ because of the next result. \emph{Translation-covariance} (with respect to  $T$) may refer to either of the equalities in \eqref{transcov} below, but we will usually retain this phrase to mean only the second equality.
\begin{thm}
  We have $T(a)(n)=-na(n-1)$. Indeed,
  \begin{align}
    \label{transcov}
Y(T(a), z)&=\partial_zY(a, z) = [T, Y(a, z)].
\end{align}
Moreover, $T$ is a derivation of $V$ in the sense that for all states $u, v\in V$ and all  $n$ we have
\begin{align*}
T(u(n)v) &= T(u)(n)v+u(n)T(v).
\end{align*}
\end{thm}
\begin{proof}
Take $v=\bone$ and $t=-2$ in the associator formula \eqref{assocform} to get
\begin{align}\label{assoc1}
 (u(-2)\bone)(s)w &= \sum_{i= 0}^{j} (-1)^i\binom{-2}{i} \left\{u(-2-i)\bone(s+i)w-\bone(s-2-i)u(i)w \right\}.
\end{align}
Now use Theorem \ref{thmvac}. The expression under the summation is nonzero in only two possible cases, namely $s+i=-1$ and $s-i=1$, and these cannot occur simultaneously.\ In the first case \eqref{assoc1} reduces to
 \begin{align}\label{specform}
  (u(-2)\bone)(s)w + s u(s-1)w  &=0.
  \end{align}
This is the first stated equality in \eqref{transcov}). In the second case, when $s-i=1$, we get exactly the same conclusion by a similar argument. This proves \eqref{transcov} in all cases.

Next we will prove that $T$ is a derivation as stated in the Theorem.\ Use the  creativity axiom and the associator formula \eqref{assocform} with $s=-2$ to see that
\begin{align*}
(u(t)v)(-2)\bone &= \sum_{i= 0}^{j} (-1)^i\binom{t}{i} u(t-i)v(-2+i)\bone = u(t)v(-2)\bone- tu(t-1)v,
  \end{align*}
Using the special case \eqref{specform} of \eqref{transcov} that has already been established, the previous display reads $Tu(t)v=  u(t)T(v) + T(u)(t)v$. This proves the derivation property of $T$.

Finally, we have
\begin{align*}
[T, a(n)]w &= T(a(n)w)-a(n)T(w)= T(a)(n)w
\end{align*}
by the derivation property of $T$. This is equivalent to the second equality
\[Y(T(a), z)=[T, Y(a, z)]\]
in \eqref{transcov}. We have now proved all parts of the Theorem.
\end{proof}

\subsection{Statement of the converse}
We have shown that a $p$-adic vertex algebra $(V, Y, \bone)$ consists, among other things, of a set of \emph{mutually local}, \emph{translation-covariant}, \emph{creative}, $p$-adic fields on $V$.\ Our next goal is to prove a converse to this statement.\ Specifically,

\begin{thm}
  \label{thmconverse}
  Let the quadruple $(V, Y, \bone, T)$ consist of a $p$-adic Banch space $V$; a continuous $\QQ_p$-linear map $Y: V\rightarrow \cF(V)$, notated $a \mapsto Y(a, z) \df \sum_{n\in\ZZ} a(n)z^{-n-1}$; a distinguished state $\bone\in V$; and a linear endomorphism $T\in \Endo(V)$ satisfying $T(\bone)=0$.

  Suppose further that:
  \begin{enumerate}
  \item[(a)] any pair of fields $Y(a, x), Y(b, y)$ are mutually local in the sense of \eqref{locform};
  \item[(b)] they are creative with respect to $\bone$ in the sense of Definition \ref{d:pva} (1);
  \item[(c)] and they are translation covariant with respect to $T$ in the sense that the second equality of \eqref{transcov} holds.
  \end{enumerate}
Then the triple $(V, Y, \bone)$ is a $p$-adic vertex algebra as in Definition \ref{d:pva}, and $T$ is the canonical derivation.
\end{thm}

We will give the proof of Theorem \ref{thmconverse} over the course of the next few Sections. Given that $(V, Y, \bone)$ \emph{is} a $p$-adic vertex algebra, it is easy to see that $T$ is its canonical derivation. For, by translation covariance, we get
\[
T(Y(a, z)\bone) = \partial_zY(a, z)\bone,
\]
so that $T(a)=T(a(-1)\bone)=a(-2)\bone$, and this is the definition of the canonical derivation \eqref{Tdef}. What we have to prove is that the Jacobi identity from Definition \ref{d:pva} is a consequence of the assumptions in Theorem \ref{thmconverse}.

\begin{lem}
  \label{Jequiv}
  The Jacobi identity is equivalent to the conjunction of the associativity formula \eqref{assocform} and the locality formula \eqref{locform}.
\end{lem}
\begin{proof}
  We have already seen in Subsection \ref{SSacloc} that the Jacobi identity implies the associativity and locality formulas.\ As for the converse, fix states $a, b, c\in V$, $r, s, t\in\ZZ$ and introduce the  notation:
  \begin{align*}
    A(r, s, t)\df&\sum_{i=0}^{\infty} \binom{r}{i} (a(t+i)b)(r+s-i)c,\\
B(r, s, t)\df&\sum_{i=0}^{\infty} (-1)^i\binom{t}{i} a(r+t-i)b(s+i)c,\\
C(r, s, t)\df&\sum_{i=0}^{\infty}    (-1)^{t+i} \binom{t}{i}b(s+t-i)a(r+i)c
\end{align*}
In these terms, the Jacobi identity is just the assertion that for all $a, b, c$ and all $r, s, t$ we have
\begin{align}\label{JIequiv}
A(r, s, t)&=B(r, s, t)-C(r, s, t).
\end{align}
On the other hand, in Subsection \ref{SSacloc} we saw that the associativity formula is just the case $r=0$ of \eqref{JIequiv}. Now use the standard formula
\begin{align*}
\binom{m}{n}&= \binom{m-1}{n} + \binom{m-1}{n-1}
\end{align*}
to see that $A(r+1, s, t)= A(r, s+1, t)+A(r, s, t+1)$, and similarly for $B$ and $C$.\ Consequently, \eqref{JIequiv} holds for all $r\geq0$ and all $s$, $t\in\ZZ$.

Now we invoke the locality assumption. In fact, we essentially saw in Subsection \ref{SSacloc} that locality is equivalent to the statement that for any $\veps>0$ there is an integer $t_0$ such that for all $t\geq t_0$ we have 
\begin{align}\label{JIineq}
\abs{A(r, s, t)-B(r, s, t)+C(r, s, t)}<\veps,
\end{align}
and we assert that \eqref{JIineq} holds uniformly for all $r$, $s$, $t$. We have explained that it holds for all $t\geq t_0$ and all $r\geq0$, so if there is a triple  $(r, s, t)$ for which it is false
then there is a triple for which $r+t$ is \emph{maximal}. But we have
\begin{align*}
&\abs{A(r, s, t)-B(r, s, t)+C(r, s, t)} \\
=&\left|A(r+1, s-1, t)-A(r, s-1, t+1) -B(r+1, s-1, t)+B(r, s-1, t+1) \right.\\
  & \left. +C(r+1, s-1, t)-C(r, s-1, t+1)\right|\\
  <& \veps
\end{align*}
by the strong triangle inequality. This completes the proof that \eqref{JIineq} holds for all $r$, $s$, $t$ and all $a$, $b$, $c\in V$. Now because $V$ is complete we can invoke Remark \ref{JIrmk} to complete the proof of the Lemma.
\end{proof}

\subsection{Residue products}
Lemma \ref{Jequiv} facilitates reduction of the proof of Theorem \ref{thmconverse} to the assertion that \emph{locality} implies \emph{associativity}.\ With this in mind we introduce \emph{residue products of fields} following \cite{LianZuckerman}. To ease notation, in the following for a state $a\in V$, we write $a(z)=Y(a, z)=\sum_n a(n)z^{-n-1}$, etc. For states $a$, $b$, $c\in V$ and any integer $t$, we define the $t^{th}$ residue product of $a(z)$ and $b(z)$ as follows:
\begin{align}\label{rpmode}
(a(z)_tb(z))(n) \df& \sum_{i=0}^{\infty} (-1)^i\binom{t}{i} \left\{a(t-i)b(n+i)-(-1)^tb(t+n-i)a(i)\right\}.
\end{align}
To be clear, $a(z)_tb(z)$ is defined to have a modal expansion with the $n^{th}$ mode being the expression \eqref{rpmode} above. The holistic way to write \eqref{rpmode} is as follows:
\begin{align}
  \label{rpalt}
a(z)_tb(z)=&\Res_y (y-z)^t a(y)b(z)-(-1)^t\Res_y (z-y)^tb(z)a(y).
\end{align}
It is important to add that here we are employing the convention\footnote{In general, expand $(z-w)^n$ as a power series in the second variable. This is inconsequential if $n\geq 0$.} that $(y-z)^t$ is expanded as a power series in $z$, whereas $(z-y)^t$ is expanded as a power series in $y$.

Notice that in a $p$-adic vertex algebra, the associativity formula \eqref{assocform} may be provocatively reformulated in the following way:
\begin{align}
  \label{rpproduct}
Y(u(t)v, z)w&= Y(u, z)_tY(v, z).
\end{align}

In the present context we do not know that $V$ is a vertex algebra: indeed, it is equation \eqref{rpproduct} that we are in the midst of proving on the basis of locality alone! Nevertheless, it behooves us to scrutinize residue products. First we have
\begin{lem}
  \label{lemfd}
  Given $a(z)$, $b(z)\in\cF(V)$, we have $a(z)_tb(z)\in\cF(V)$ for every $t \in \ZZ$. Moreover,
  \[
  \abs{a(z)_tb(z)} \leq \abs{a(z)}\abs{b(z)}.
  \]
\end{lem}
\begin{proof}
  From equation \eqref{rpmode} we obtain the bound
  \begin{align*}
    \abs{a(z)_tb(z)(n)} \leq& \sup_{i\geq 0}\left(\sup\{\abs{a(t-i)b(n+i)},\abs{b(t+n-i)a(i)}\}\right)\\
  \leq& \sup_{i\geq 0}\left(\sup\{\abs{a(t-i)}\abs{b(n+i)},\abs{b(t+n-i)}\abs{a(i)}\}\right)\\
  \leq&  \sup_{u\in\ZZ}\abs{a(u)} \cdot \sup_{v\in\ZZ}\abs{b(v)} = \abs{a(z)}\abs{b(z)}.
  \end{align*}
  
This establishes the desired inequality, and it shows that property (1) of Definition \ref{d:pfield} holds with $M= \abs{a(z)}$. Similarly, since we know that $\lim_{n\to \infty} a(n)c=\lim_{n\to \infty} b(n)c=0$ then
\begin{align*}
\abs{(a(z)_tb(z))(n)c} &= \sup_i \abs{a(t-i)b(n+i)c-(-1)^tb(t+n-i)a(i)c} \\
  &\leq  M\sup_i \left\{ \abs{b(n+i)c}, \abs{b(t+n-i)c}\right\}\\
  &\rightarrow_n 0.
\end{align*}
which establishes property (2) of Definition \ref{d:pfield}. The Lemma is proved.
\end{proof}

\begin{rmk}
Lemma \ref{lemfd} has the following interpretation: identify $V$ with its image in $\cF(V)$ under the vertex operator map $a\mapsto a(z)$, and define a new vertex operator structure on these fields using the  residue products. Then Lemma \ref{lemfd} says that these new fields are submultiplicative in the sense of Definition \ref{d:pfieldsn}.  This is so even though we have not assumed that $V$ is submultiplicative, cf. Corollary \ref{corVW} for further discussion of this point.
\end{rmk}

Now recall that we are assuming the hypotheses of Theorem \ref{thmconverse} that $a(z)$, $b(z)$ are creative with respect to $\bone$ and, in particular, $b(z)$ creates the state $b$ in the sense that $b(z)\bone=b+O(z)$.
\begin{lem}\label{lemresprodcreate} For all integers $t$, $a(z)_tb(z)$ is creative with respect to $\bone$ and creates the state $a(t)b$.
\end{lem}
\begin{proof} This is straightforward, for we have
\begin{align*}
(a(z)_tb(z))(n)\bone =& \sum_{i=0}^{\infty} (-1)^i\binom{i}{i}\left\{a(t-i)b(n+i)\bone-(-1)^tb(t+n-i)a(i)\bone\right\}\\
=& \sum_{i=0}^{\infty} (-1)^i\binom{t}{i}\left\{a(t-i)b(n+i)\bone\right\}
\end{align*}
and this vanishes if $n\geq0$ because $b(z)$ is creative with respect to $\bone$.\ Finally, if $n=-1$, for similar reasons the last display reduces to $a(t)b$, and the Lemma is proved.
\end{proof}

\subsection{Further properties of residue products}
\begin{lem}\label{thmDongslem} Let $a, b, c\in V$, so that $a(z)$, $b(z)$, $c(z)$ are mutually local $p$-adic fields. Then for any integer $t$, $a(z)_tb(z)$ and $c(z)$ are also mutually local.
\end{lem}
\begin{proof}
  We have to show that $\lim_{n\to\infty} (x-y)^n[a(x)_tb(x), c(y)]=0$. We have
\begin{align*}
& \lim_{n\to\infty} (x-y)^n[a(x)_tb(x), c(y)]\\
=&\lim_{n\to\infty} (x-y)^n   \left\{\Res_w (w-x)^t [a(w)b(x), c(y)] - (-1)^t\Res_w (x-w)^t[b(x)a(w), c(y)]\right\}\\
=&\lim_{n\to\infty} (x-y)^n  \Res_w (w-x)^t \left\{a(w)[b(x), c(y)] + [a(w), c(y)]b(x) \right\}-\\
&(-1)^t  \lim_{n\to\infty} (x-y)^n \Res_w (x-w)^t  \left\{ b(x)[a(w), c(y)]+ [b(x), c(y)]a(w)\right\}\\
=&\lim_{n\to\infty} (x-y)^n \Res_w (w-x)^t \left\{ [a(w), c(y)]b(x) \right\}-\\
&(-1)^t  \lim_{n\to\infty} (x-y)^n \Res_w (x-w)^t  \left\{ b(x)[a(w), c(y)]\right\},
\end{align*}
where we used the fact that $b(z)$ and $c(z)$ are mutually local $p$-adic fields.

Note that if $t\geq0$ then $(w-x)^t=(-1)^t(x-w)^t$ and the last expression vanishes as required.\ So we may assume without loss that $t<0$.

Now use the identity, with $n\geq m\geq 0$,
\begin{align*}
(x-y)^n =(x-y)^{n-m} (x-w+w-y)^{m}=(x-y)^{n-m}\sum_{i=0}^{m} \binom{m}{i} (x-w)^i(w-y)^{m-i}
\end{align*}
Pick any $\veps>0$.\ There is a positive integer $N_0$ such that $\left\vert (w-y)^N[a(w), b(y)]\right\vert < \veps$ whenever $N\geq N_0$.\ We will
choose $m\df N_0-t$, which is nonnegative because $t<0$. The first summand in the last displayed equality is equal to 
\begin{align*}
&  \sum_{i=0}^{m} (-1)^i\binom{m}{i}  \lim_{n\to\infty} (x-y)^{n-m}\Res_w   (w-x)^{t+i}(w-y)^{m-i} \left\{ [a(w), c(y)]b(x) \right\}\\
=&\sum_{i=0}^{-t}``" +\sum_{i=-t+1}^m``",
\end{align*}
where the quotation marks indicate that we are summing the same general term as in the left-side of the equality. We get a very similar formula for the second summand. 

Repeating the use of a previous device, if $-t+1\leq i$ then $t+i\geq 0$, so that $(-1)^{t+i}(w-x)^{t+i}=(x-w)^{t+i}$.\ From this it follows that in our expression for $\lim_{n\to\infty} (x-y)^n[a(x)_tb(x), c(y)]$ the two sums over the range $-t+1\leq i\leq m$ in fact \emph{cancel}. Hence we find that
\begin{align*}
&\lim_{n\to\infty} (x-y)^n[a(x)_tb(x), c(y)]\\
=&\sum_{i=0}^{-t}  (-1)^i\binom{m}{i}\left\{ \lim_{n\to\infty} (x-y)^{n-m}\Res_w   (w-x)^{t+i}(w-y)^{m-i}  [a(w), c(y)]b(x) \right.-\\
&\left.(-1)^{t}\lim_{n\to\infty} (x-y)^{n-m}\Res_w   (x-w)^{t+i}(w-y)^{m-i}  b(x)[a(w), c(y)]\right\}.
\end{align*}
For $i \leq -t$ we have $m-i\geq m+t=N_0$ and therefore each of the expressions $(w-y)^{m-i}[a(w), c(y)]$ have norms less than $\veps$. Therefore we see that \[\lim_{n\to\infty} (x-y)^n[a(x)_tb(x), c(y)]=0,\] and the Lemma is proved.
\end{proof}

\begin{lem}\label{rptranscov}If $a(z)$ and $b(z)$ are translation covariant with respect to $T$, then  $a(z)_tb(z)$ is translation covariant with respect to $T$ for all integers $t$.
\end{lem}
\begin{proof} This is a straightforward consequence of the next two equations, both readily checked by direct caculation:
\begin{align}\label{rptc}
[T, a(z)_tb(z)]=&[T, a(z)]_tb(z)+a(z)_t[T, b(z)]   \notag  \\
 \partial_z(a(z)_tb(z))=&(\partial_za(z))_tb(z)+ a(z)_t(\partial_zb(z))
\end{align}

\end{proof}
\subsection{Completion of the proof of Theorem \ref{thmconverse}}
We have managed to reduce the proof of Theorem \ref{thmconverse} to showing that $p$-adic locality implies associativity, i.e., the formula \eqref{rpproduct} in the format
\begin{align}\label{assocnew}
(a(t)b)(z)=&a(z)_tb(z)
\end{align}
We shall carry this out now on the basis of Lemmas \ref{lemfd} through \ref{thmDongslem}. We need one more preliminary result.

\begin{lem}\label{lemd}
  Suppose that $d(z)$ is a $p$-adic field that is translation covariant with respect to $T$;  mutually local with all $p$-adic fields $a(z)\ (a\in V)$; and creative with respect to $\bone$.\ Then $d(z)=0$ if, and only if, $d(z)$ creates $0$.
\end{lem}
\begin{proof}
  It suffices to show that if $d(z)$ creates $0$, i.e., $d(-1)\bone=0$, then $d(z)=0$.

To begin the proof, recall that by hypothesis we have $T\bone=0$.\ Then 
\begin{align*}
[T, d(z)]\bone=&T(d(z)\bone)= \sum_{n<0} T (d(n)\bone) z^{-n-1}.
\end{align*}
By translation covariance we also have
\begin{align*}
[T, d(z)]\bone = \partial_z d(z)\bone = \sum_{n<0}  (-n-1)d(n)\bone z^{-n-2}
\end{align*}
Therefore, for all $n<-1$ we have $Td(n+1)\bone = (-n-1)d(n)\bone$.\ But $d(-1)\bone=0$ by assumption.\ Therefore $d(-2)\bone=0$, and by induction we find that
for all $n<0$ we have $d(n)\bone=0$.\ This means that $d(z)\bone=0$.

Now we prove that in fact $d(z)=0$.\ For any state $u\in V$ we know that $u(z)$ creates $u$ from $\bone$.\ Moreover we are assuming that $d(z)$ and $u(z)$ are mutually local, i.e., $\lim_{n\to\infty} (z-w)^n[d(z), u(w)]=0$. Consider
 \begin{align*}
\lim_{n\to\infty}z^nd(z)u  =&  \lim_{n\to\infty}\Res_w w^{-1}(z-w)^n d(z) u(w)\bone\\
=& \lim_{n\to\infty}\Res_w w^{-1}(z-w)^n u(w)d(z)\bone=0.
 \end{align*}
 Notice that the sup-norms of $d(z)u$ and $z^nd(z)u$ are equal, since multiplication by $z^n$ just shifts indices.  Therefore, the only way that the limit above can vanish is if $d(z)u=0$. Since $u$ is an arbitrary element of $V$, we deduce that $d(z)=0$, and this completes the proof of Lemma \ref{lemd}.
\end{proof}

Turning to the proof of Theorem \ref{thmconverse}, suppose that $a, b\in V$, let $t$ be an integer, and consider $d(z)\df a(z)_tb(z)-(a(t)b)(z)$. This is a $p$-adic field by Lemma \ref{lemfd}, and it is creative by Lemma \ref{lemresprodcreate}, indeed that Lemma implies that $d(z)$ creates $0$. Now Lemmas \ref{lemresprodcreate} through \ref{rptranscov} supply the hypotheses of Lemma \ref{lemd}, so we can apply the latter to see that $d(z)=0$. In other words, $a(z)_tb(z) = (a(t)b)(z)$. Thus we have established \eqref{assocnew} and the proof of Theorem \ref{thmconverse} is complete.

We record a Corollary of the proof that was hinted at during our discussion of equation \eqref{rpproduct}.
\begin{cor}\label{corVW}
  Suppose that $(V, Y, \bone)$ is a $p$-adic vertex algebra, and let $W\subseteq \cF(V)$ be the image of $Y$ consisting of the $p$-adic fields $Y(a, z)$ for $a\in V$. If we define the $t^{th}$ product of these fields by the residue product $Y(a, z)_tY(b, z)$, then $W$ is a \emph{submultiplicative} $p$-adic vertex algebra with vacuum state $\id_V$ and canonical derivation $\partial_z$. Moreover, $Y$ induces a bijective, continuous map of $p$-adic vertex algebras $Y: V \longrightarrow W$ that preserves all products.
\end{cor}
\begin{proof}
  The set of fields $Y(a, z)$ is closed with respect to all residue products by Lemma \ref{rpproduct}, and $Y(\bone, z)=\id_V$ by Theorem \ref{thmvac}. Then equation \eqref{rpproduct} says that the state-field correspondence $Y: (V, \bone)\longrightarrow (W, \id_V)$ preserves all products.\ Since $V$ is a $p$-adic vertex algebra and $\cF(V)$ is a $p$-adic Banach space by Proposition \ref{propcFV}, then so too is $W$ because $W$ is closed by Proposition \ref{p:closedmap}.\ Next we show that $\partial _z$ is the canonical derivation for $W$.\ Certainly $\partial_z\id_V=0$,  and it remains to show that $Y(a, z)_{-2}\id_V=\partial_z Y(a, z)$. But this is just the case $t=-1$, $b(z)=\id_V$ of \eqref{rptc}. This shows that $W$ is a $p$-adic vertex algebra.\ That it is submultiplicaitve follows from \eqref{rpproduct} and Lemma \ref{lemfd}. Finally,  we assert that  $Y: V\rightarrow W$ is a continuous bijection.\ Continuity follows by definition and  injectivity is a consequence of the creativity axiom.
\end{proof}

\begin{rmk}
  It is unclear to us whether the map $Y:V\longrightarrow W$ of Corollary \ref{corVW} is always a topological isomorphism. It may be that $W$ functions as a sort of completion of $V$ that rectifies any lack of submultiplicativity in its $p$-adic fields. However, we do not know of an example where $Y$ is not a topological isomorphism onto its image.
\end{rmk}

\section{\texorpdfstring{$p$}{p}-adic conformal vertex algebras}
\label{s:conformal}

Here we develop the theory of vertex algebras containing a dedicated Virasoro vector.\ We find it convenient to divide the main construction into two halves, namely \emph{weakly conformal} and \emph{conformal} $p$-adic vertex algebras.
One of the main results is Proposition \ref{propsubVOA}.

\subsection{The basic definition}
As a preliminary, we recall that the \emph{Virasoro algebra over $\QQ_p$ of central charge $c$} is the Lie algebra with a canonical $\QQ_p$-basis consisting of $L(n)$, for $n\in \ZZ$, and a central element $\kappa$ that satisfies the relations of Definition \ref{altVir}.

\begin{rmk}
  \label{rmkqcc}
  When working over a general base ring it is often convenient to replace the central charge $c$ with the \emph{quasicentral charge} $c'$ defined by $c'\df\tfrac{1}{2}c$, so that the basic Virasoro relation reads
\begin{eqnarray}\label{Virreln2}
[L(m), L(n)]= (m-n)L(m+n) + \delta_{m+n, 0}\binom{m+1}{3}c' \kappa
\end{eqnarray}
where all coefficients, with the possible exception of $c'$, lie in $\ZZ$.
\end{rmk}

\begin{dfn}
  \label{dfnpcva}
  A \emph{weakly conformal $p$-adic vertex algebra} is a quadruple $(V, Y, \bone, \omega)$ such that $(V, Y, \bone)$ is a $p$-adic vertex algebra (cf.\ Definition \ref{d:pva})
and $\omega\in V$ is a distinguished state (called the \emph{conformal vector, or Virasoro vector}) such that
$Y(\omega, z)=:\sum_{n\in\ZZ}L(n)z^{-n-2}$ enjoys the following properties:
\begin{enumerate}
\item[(a)] $[L(m), L(n)]= (m-n)L(m+n) + \tfrac{1}{12}\delta_{m+n, 0}(m^3-m)c_V \id_V$ for some $c=c_V\in \QQ_p$ called the central charge of $V$,
\item[(b)] $[L(-1), Y(v, z)]=\partial_z Y(v, z)$ for $v\in V$.
\end{enumerate}
This completes the Definition.
\end{dfn}

\begin{rmk} We are using the symbols $L(n)$ in two somewhat different ways, namely as elements in an abstract Virasoro Lie algebra in (\ref{altVir}) and as modes $\omega(n+1)$ of the vertex operator $Y(\omega, z)$ in part (a) of Definition \ref{dfnpcva}.\ This practice is traditional and should cause no confusion.\ Thus the meaning of (a) is that the modes of the conformal vector close on a representation of the Virasoro algebra of central charge $c_V$ as operators on $V$ such that the central element $\kappa$
acts as $cId_V$.
\end{rmk}

\begin{rmk} Part (b) of Definition \ref{dfnpcva} is called \emph{translation covariance}.\ Comparison with Subsection \ref{SStranscov} shows that the meaning of part (b) is that $L(-1)$ is none other than the  canonical derivation $T$ of the vertex algebra $(V, Y, \bone)$.
\end{rmk}

\subsection{Elementary properties of weakly conformal $p$-adic vertex algebras}
We draw some basic conclusions from Definition \ref{dfnpcva}.\ These are well-known in algebraic VOA theory, although the context is slightly different here.\ For an integer $k\in\ZZ$ we set
\[
V_{(k)}\df \{v\in V\mid L(0)v=kv\}
\]

\begin{lem}
  \label{lem1omega}
 We have $\bone\in V_{(0)}$ and $\omega\in V_{(2)}$.
\end{lem}
\begin{proof}
  First apply the creation axiom to see that
\begin{eqnarray*}
Y(\omega, z)\bone=\sum_{n\in\ZZ}L(n)\bone z^{-n-2}= \omega+ O(z),
\end{eqnarray*}
in particular $L(0)\bone=0$, that is $\bone\in V_{(0)}$. This proves the first containment stated in the Lemma and it also shows that $L(-2)\bone=\omega$.

Now use the Virasoro relations to see that 
\[
L(0)\omega=L(0)L(-2)\bone = [L(0), L(-2)]\bone=2L(-2)\bone=2\omega,
\]
that is $\omega\in V_{(2)}$, which is the other needed containment.\ The Lemma is proved.
\end{proof}

\begin{lem}
  \label{lemgraded}
  Suppose that $u\in V_{(m)}$.\ Then for all integers $\ell$ and $k$ we have
\[
u(\ell): V_{(k)}\rightarrow V_{(k+m-\ell-1)}.
\]
\end{lem}
\begin{proof}
  Let $v\in V_{(k)}$. We have to show that $L(0)u(\ell)v=(k+m-\ell-1)u(\ell)v$. With this in mind, first consider the commutator formula Proposition \ref{p:commutator} with $\omega$ in place of $u$ and $r=0$. Bearing in mind that $L(-1)=\omega(0)$ we obtain
\[
[L(-1), v(s+1)]= (L(-1)v)(s),
\]
and a comparison with the translation covariance axiom (b) above then shows that
\begin{align}\label{tc1}
(L(-1)v)(s)&=-sv(s-1).
\end{align}
Now we calculate
\[
L(0)u(\ell)v=[L(0), u(\ell)]v +u(\ell)L(0)v = \sum_{i=0}^1 (L(i-1)u)(1+\ell-i)v +ku(\ell)v
\]
where we made another application of the commutator formula to get the first summand.\ Consequently, by \eqref{tc1} we have
\begin{align*}
L(0)u(\ell)v&= (L(-1)u)(1+\ell)v+(L(0)u)(\ell)v+ku(\ell)v \\
&=-(\ell+1)u(\ell)v+mu(\ell)v+ku(\ell)v,
\end{align*}
which is what we needed.
\end{proof}

\subsection{$p$-adic vertex operator algebras}
\begin{dfn}
  \label{pcva1}
  A \emph{conformal $p$-adic vertex  algebra} is a weakly conformal $p$-adic vertex algebra $(V, Y, \bone, \omega)$ as in Definition \ref{dfnpcva} such that the integral part of the spectrum of $L(0)$ has the following two additional properties:
\begin{enumerate}
\item[(a)] each eigenspace $V_{(k)}$ is a finite-dimensional $\QQ_p$-linear space,
\item[(b)] there is an integer $t$ such that $V_{(k)}=0$ for all $k<t$.
\end{enumerate}
\end{dfn}

We now have
\begin{prop}
  \label{propsubVOA}
  Suppose that $(V,Y, \bone, \omega)$ is a conformal $p$-adic  vertex algebra, and let $U\df \oplus_{k\in \ZZ} V_{(k)}$ be the sum of the integral $L(0)$-eigenspaces. Then $(U, Y, \bone, \omega)$ is an algebraic
VOA over $\QQ_p$.
\end{prop}
\begin{proof}
  The results of the previous Subsection apply here, so in particular $U$ contains $\bone$ and $\omega$ by Lemma \ref{lem1omega}. Moreover, if $u\in U$, then $Y(u, z)$ belongs to $\pseries{\Endo(U)}{z, z^{-1}}$ thanks to Lemma \ref{lemgraded}.\ More is true: because $V_{(k)}=0$ for $k<t$, the same Lemma implies that $u(\ell)v=0$ whenever $u\in V_{(m)}$, $v\in V_{(k)}$ and $\ell >k+m-1$. Therefore, each $Y(u, z)$ is an algebraic field on $U$.

Because $(V, Y, \bone)$ is a $p$-adic subVA of $V$, then the vertex operators $Y(u, z)$ for $u\in U$ satisfy the algebraic Jacobi identity, and therefore
$(U, Y, \bone, \omega)$ is an algebraic VOA over $\QQ_p$.
\end{proof}

\begin{dfn}
  \label{dfnpvoa1}
  A \emph{$p$-adic VOA} is a conformal $p$-adic vertex algebra $(V, Y, \bone, \omega)$ as in Definition \ref{pcva1} such that $V=\widehat{U}$.
\end{dfn}

As an immediate consequence of the preceding Definition and Proposition \ref{propsubVOA}, we have
\begin{cor}
  Suppose that $(V,Y, \bone, \omega)$ is a $p$-adic conformal vertex algebra, and let $U\df \oplus_{k\in \ZZ} V_{(k)}$ be the sum of the $L(0)$-eigenspaces. Then the completion $\widehat{U}$ of $U$ in $V$ is a $p$-adic VOA. 
$\hfill\Box$
\end{cor}

\subsection{Examples} We discuss some elementary examples of $p$-adic VOAs of a type that are familiar in the context of algebraic $k$-vertex algebras.\
The first two are degenerate in the informal sense that the conformal vector $\omega$ is equal to $0$

\begin{ex}
  \emph{$p$-adic topological rings.}\\
Suppose that $V$ is a commutative, associative ring with identity $1$ that is also a  $p$-adic Banach space.\ Concerning the continuity of multiplication, we only need to assume that each multiplication $a: b\mapsto ab$, for $a, b\in V$, is bounded, i.e., $\abs{ab}\leq M\abs{b}$ for a nonnegative constant $M$ depending on $a$ but independent of $b$. Define the vertex operator $Y(a, z)$ to be multiplication by $a$. We claim that $(V, Y, 1)$ is a $p$-adic vertex algebra. Indeed, $\abs{Y(a, z)}=\abs{a}$, so $Y$ is bounded and hence continuous.
The commutativity of multiplication amounts to $p$-adic locality, cf.
Definition \ref{d:locality}, and the remaining assumptions of Theorem \ref{thmconverse} are readily verified (taking $T=0$). Then this result says that indeed
$(V, Y, 1)$ is a $p$-adic vertex ring. Actually, it is a $p$-adic VOA (with $\omega=0$) as long as $V=V_{(0)}$ is finite-dimensional over $\QQ_p$.
\end{ex}

\begin{ex}
\emph{Commutative $p$-adic Banach algebras.}\\
 This is essentially a special case of the preceding Example 1.\ In a $p$-adic Banach ring $V$ with identity, one has (by definition)
that multiplication is  submultiplicative, i.e., $\abs{ab} \leq \abs{a}\abs{b}$.\ Thus in the previous Example, the vertex operator $Y(a, z)$ is submultiplicative (compare with Definition \ref{dfnpadicfd}).
\end{ex}

\begin{ex}
\emph{The $p$-adic Virasoro VOA.}\\
The construction of the Virasoro vertex algebra over $\CC$ can be basically reproduced working over $\ZZ_p$ \cite{Mason1}. We recall some details here.\

We always assume that the quasicentral charge $c'$ as defined in \eqref{Virreln2} lies in $\ZZ_p$. Form the highest weight module, call it $W$, for the Virasoro algebra over $\QQ_p$ that is generated by a state $v_0$ annihilated by all modes $L(n)$, for $n\geq 0$, and on which the central element $K$ acts as the identity. 

By definition (and the Poincar\'{e}-Birkhoff-Witt theorem) then, $W$ has a basis consisting of the states $L(-n_1)\cdots L(-n_r)v_0$ for all sequences  $n_1\geq n_2\geq \ldots \geq n_r\geq 1$.\ Thanks to our assumption that $c'\in \ZZ_p$ it is evident from \eqref{Virreln2} that the very same basis affords a module for the Virasoro Lie ring over $\ZZ_p$.\ We also let $W$ denote this $\ZZ_p$-lattice.\ Let $W_1\subseteq W$ be the $\ZZ_p$-submodule generated by $L(-1)v_0$.\ Then the quotient module $W/W_1$ has a canonical $\ZZ_p$-basis consisting of states 
$L(-n_1)\cdots L(-n_r)v_0+W_1$ for all sequences  $n_1\geq n_2\geq \ldots\geq n_r\geq 2$.\

The Virasoro vertex algebra over $\ZZ_p$ with quasicentral charge $c'$ has (by Definition) the Fock space $V\df W/W_1$ and the  vacuum element is
$\bone\df v_0+W_1$. By construction $V$ is a module for the Virasoro ring over $\ZZ_p$, and the vertex operators $Y(v, z)$ for, say $v=L(-n_1)\cdots L(-n_r)v_0+W_1$, may be defined in terms of the actions of the $L(n)$ on $V$.\ For further details see \cite[Section 7]{Mason1}, where it is also proved (Theorem 7.3, loc. cit) that when so defined, $(V, Y, \bone)$ is an algebraic vertex algebra over $\ZZ_p$.\ Indeed, it is shown (loc.\ cit, Subsection 7.5) that $(V, Y, \bone, \omega)$ is a VOA over $\ZZ_p$ in the sense defined there.\ But this will concern us less because we want to treat our algebraic vertex algebra $(W/W_1, Y, \bone)$ over $\ZZ_p$  $p$-adically.

Having gotten this far, the next step is almost self-evident. We introduce the completion
\[
V'\df \varprojlim V/p^kV
\]

\begin{thm}
  \label{t:virasoro}
  The completion $V'$ is a $p$-adic VOA in the sense of Definition \ref{dfnpvoa1}.
\end{thm}
\begin{proof}
  It goes without saying that $V' = (V', Y, \bone, \omega)$ where $V',Y$ and $\bone$ have already been explained.\ We define $\omega\df L(-2)v_0+W_1=L(-2)\bone\in V$.\ Implicit in \cite[Theorem 7.3]{Mason1} are the statements that $\omega$ is a Virasoro state in $V$, and that its mode $L(0)$ has the
properties (a) and (b) of Definition \ref{dfnpcva}. Thus $(V', Y, \bone, \omega)$ is a weakly conformal $p$-adic VOA.

Now because $V$ is an algebraic VOA then $V=\oplus V_{(k)}$ has a decomposition into $L(0)$-eigenspaces as indicated, with $V_{(k)}$ a finitely generated free $\ZZ_p$-module and $V_{(k)}=0$ for all small enough $k$. By Proposition \ref{propint}  $V$ is also the sum of the integral $L(0)$-eigenspaces in $V'$. So in fact $(V', Y, \bone, \omega)$ is a conformal $p$-adic vertex algebra. Since by definition $V'$ is the completion of $V$ then $(V', Y, \bone, \omega)$ is a $p$-adic VOA according to Definition \ref{dfnpvoa1}.\ The Theorem is proved.
\end{proof}
\end{ex}

\section{Completion of an algebraic vertex operator algebra}
\label{s:completion}
We now discuss the situation of completing algebraic VOAs. The structure that arises was the model for our definition of $p$-adic VOA above. This section generalizes Theorem \ref{t:virasoro} above.

Let $V=(V, Y, \bone)$ be an algebraic VOA over $\Qp$. Assume further that $V$ is equipped with a nonarchimedean absolute value $\abs{\cdot}$ that is compatible with the absolute value on $\Qp$ in the sense that if $\alpha \in \Qp$ and $v \in V$ then $\abs{\alpha v} = \abs{\alpha}\abs{v}$, and such that $\abs{\bone}\leq 1$. We do not necessarily assume that $V$ is complete with respect to this absolute value.\ In order for the completion of $V$ to inherit the structure of a $p$-adic VOA, it will be necessary to assume some compatibility between the VOA axioms and the topology on $V$.\ To this end we make the following definition.
\begin{dfn}\label{compatabsval} We say that $\abs{\cdot}$ is \emph{compatible} with $V$ if the following properties hold:
\begin{enumerate}
\item for each $a \in V$, the algebraic vertex operator $Y(a,z)$ belongs to $\cF(V)$, i.e., it is a $p$-adic field in the sense of Definition \ref{d:pfield}. In particular, each mode $a(n)$ is a bounded and hence continuous endomorphism of $V$.
\item the association $a \mapsto Y(a,z)$ is continuous when $\cF(V)$ is given the topology derived from the sup-norm.
\end{enumerate}
\end{dfn}
These conditions are well-defined even though $V$ is not assumed to be complete.\ The basic observation is the following:
\begin{prop}
  \label{p:compatabsval}
Assume that $\abs{\cdot}$ is compatible with $V=(V, Y, \bone)$ in the sense of Definition \ref{compatabsval}, and let $V'$ denote the completion of $V$ with respect to $\abs{\cdot}$. Then $V'$ has a natural structure of $p$-adic VOA.
\end{prop}
\begin{proof}
  Let $a\in V'$ so that by definition we can write $a=\lim_{n\to\infty} a_n$ for $a_n\in V$. We first explain how to define $Y(a,z)$ by taking a limit of the algebraic fields $Y(a_n,z)$, and then we show that this limit is a $p$-adic field on $V'$.
  
 The modes of each $a_n$ are continuous, and so they extend to continuous endomorphisms of $V'$, and indeed, since $Y(a_n,z) \in \cF(V)$ by condition (1) of Definition \ref{compatabsval}, we can naturally view this as a $p$-adic field in $\cF(V')$ by continuity. Next, by continuity of the map $b \mapsto Y(b,z)$ on $V$, and since $\cF(V')$ is complete by Proposition \ref{propcFV}, we deduce that $Y(a,z) = \lim_{n\to \infty}Y(a_n,z)$ is also contained in $\cF(V')$. The extended map $Y \colon V' \to \cF(V')$ inherits continuity similarly. 

  Next, it is clear that the creativity axiom is preserved by taking limits. Likewise, the Jacobi identity is preserved under taking limits, as it is the same equality as in the algebraic case. Finally, the Virasoro structure of $V'$ is inherited from $V$, and similarly for the (continuous) direct sum decomposition.
\end{proof}

The following Proposition is straightforward, but we state it explicitly to make it clear that the continuously-graded pieces of the $p$-adic completion of an algebraic VOA do not grow in dimension:
\begin{prop}\label{propint}
  Let $V = \bigoplus_{n\geq N} V_{(n)}$ denote an algebraic VOA endowed with a compatible nonarchimedean absolute value, and let $V' = \widehat{\bigoplus}_{n\geq N} V_{(n)}'$ denote the corresponding completion ($p$-adic VOA). Then $V_{(n)} = V_{(n)}'$ and
  \[
    V'_{(n)} = \{v \in V' \mid L(0)v = nv\}.
  \]
\end{prop}
\begin{proof}
  Since each $V_{(n)}$ is a finite-dimensional $\Qp$-vector space by hypothesis, it is automatically complete with respect to any nonarchimedean norm. Therefore, the completion $V'$ consists of sums $v = \sum_{n\geq N} v_n$ with $v_n \in V_{(n)}$ for all $n$, such that $v_n \to 0$ as $n \to \infty$. Moreover $v = 0$ if and only if $v_n = 0$ for all $n$. Since $L(0)$ is assumed to be continuous with respect to the topology on $V$ (since all modes are assumed continuous) we deduce that
  \[
  (L(0)-n)v = \sum_{m\geq N} (L(0)-n)v_m = \sum_{m\geq N}(m-n)v_m.
\]
The preceding expression only vanishes if $v \in V_n$. Thus $L(0)$ does not acquire any new eigenvectors with integer eigenvalues in $V'$, which concludes the proof.
\end{proof}

Now we wish to discuss one way that $p$-adic VOAs arise. It remains an open question to give an example of a $p$-adic VOA that does not arise via completion of an algebraic VOA with respect to some absolute value.

Let $U = \bigoplus_k U_{(k)}$ denote an algebraic vertex algebra over $\Zp$ in the sense of \cite{Mason1} and Section \ref{s:primer}, but equipped with an integral decomposition as shown, where each graded piece $U_{(k)}$ is assumed to be a free $\Zp$-module of finite rank.\ Then  $V \df U \otimes_{\ZZ_p}\QQ_p$ has the structure of an algebraic vertex algebra, and suppose moreover that $V$ has a structure of algebraic VOA compatible with this vertex algebra structure.\ As usual, $V=\oplus_k
V_{(k)}$ is the decomposition of $V$ into $L(0)$-eigenspaces. Note that the conformal vector $\omega \in V$ may not be integral, that is, it might not be contained in $U$.\
Suppose, however, that the gradings are compatible, so that $U_{(k)} = U \cap V_{(k)}$. Choose $\Zp$-bases for each submodule $U_{(k)}$, and endow $V$ with the corresponding sup-norm: that is, if $(e_j)$ are basis vectors, then $\abs{\sum \alpha_je_j} = \sup_j \abs{\alpha_j}$. 

\begin{prop}
  \label{p:integralcompletion}
  Let $V = U\otimes_{\Zp}\Qp$, equipped with the induced sup-norm as above. Then the sup-norm is compatible with the algebraic VOA structure.\ Thus, the completion $V'$ is a $p$-adic VOA.
\end{prop}
\begin{proof}
  Each $a \in V$ is a finite $\Qp$-linear combination of elements in the homogeneous pieces $U_{(k)}$ of the integral model. The modes of elements in $U_{(k)}$ are direct sums of maps between finite-rank free $\Zp$-modules, and so they are thus uniformly bounded by $1$ in the operator norm. It follows that the modes appearing in any $Y(a,z)$ are uniformly bounded. Further, for each $b \in V$, the series $Y(a,z)b$ has a finite Laurent tail since $V$ is an algebraic VOA. Therefore each $Y(a,z)$ is a $p$-adic field, as required by condition (1) of Definition \ref{compatabsval}.

  It remains to show that the state-field correspondence $V \to \cF(V)$ is continuous. We would like to use the conclusion of Proposition \ref{p:closedmap} below, but that result assumed that we had a $p$-adic VOA to begin with. To avoid circularity, let us explain why the conclusion nevertheless applies here: first, since $\bone$ is from the underlying algebraic VOA, we still have $Y(\bone,z) = \id_{V}$. This is all that is required to obtain $\abs{\bone}\leq \abs{Y(\bone,z)}$, as in the proof of Corollary \ref{p:closedmap}. Then using this, and since $a(-1)\bone=a$ for $a \in V$ since $V$ is an algebraic VOA, we deduce that
  \begin{equation}
    \label{eq:cont}
    \abs{a} \leq \abs{Y(a,z)}
  \end{equation}
  for all $a \in V$, as in the proof of Proposition \ref{p:closedmap}.

  Now, returning to the continuity of the map $V \to \cF(V)$, write $a \in V$ as $a = \sum_j \alpha_je_j$ where $\alpha_j \in \Qp$ and $e_j\in U_{(k_j)}$ are basis vectors used to define the sup-norm on $V$. In particular, we have $\abs{e_j(n)}\leq 1$ for all $j$ and $n$, since $e_j(n)$ is defined over $\ZZ_p$ by hypothesis. Then
  \[
    \abs{Y(a,z)} = \sup_n\abs{a(n)} \leq \sup_n\sup_j\abs{\alpha_j}\abs{e_j(n)} \leq \sup_j\abs{\alpha_j} = \abs{a}.
  \]
Coupled with the previously establish inequality \eqref{eq:cont}, this implies that under the present hypotheses we have $\abs{Y(a,z)}=\abs{a}$. Thus, $Y$ is an isometric embedding, and this confirms in particular that the state-field correspondence is continuous, as required by condition (2) of Definition \ref{compatabsval}. Therefore we may apply Proposition \ref{p:compatabsval} to complete the proof.
\end{proof}
\begin{rmk}
We point out that the integral structure above was necessary to ensure a uniform bound for the modes $a(n)$, independent of $n\in\ZZ$, as required by the definition of a $p$-adic field.
\end{rmk}

\begin{rmk}
  There is a second, equivalent way to obtain a $p$-adic VOA from $U$ (cf.\ Subsection \ref{SSprojlim}).\ First observe that the $\Zp$-submodules $p^nU$ are  $2$-sided ideals defined over $\Zp$. We set
\[
  U' = \varprojlim U/p^nU.
\]
This is a limit of vertex rings over the finite rings $\ZZ/p^n\ZZ$ and $U'$ has a natural structure of $\ZZ_p$-module. Let $V' = U'\otimes_{\ZZ_p}\QQ_p$. Then this carries a natural structure of $p$-adic VOA that agrees with the sup-norm construction above.
\end{rmk}

\section{Further remarks on \texorpdfstring{$p$}{p}-adic locality}
\label{s:locality}
The Goddard axioms of Section \ref{s:goddard} illustrate that if $Y(a,z)$ and $Y(b,z)$ are $p$-adic vertex operators, then they are \emph{mutually $p$-adically local} in the sense that
\[
  \lim_{n\to\infty} (x-y)^n[Y(a,x),Y(b,y)]=0.
\]
In this section we adapt some standard arguments from the theory of algebraic vertex algebras on locality to this $p$-adic setting.

Let us first analyze what sort of series $Y(a,x)Y(b,y)$ is. We write
\[
  Y(a,x)Y(b,y) = \sum_{m,n\in\ZZ} a(n)b(m)x^{-n-1}y^{-m-1}.
\]
We wish to show that:
\begin{enumerate}
\item there exists $M \in \RR_{\geq 0}$ such that $\abs{a(n)b(m)c} \leq M\abs{c}$ for all $c \in V$ and all $n,m\in\ZZ$;
\item $\lim_{m,n\to\infty} a(n)b(m)c=0$ for all $c \in V$.
\end{enumerate}
Property (1) is true since $a(n)$ and $b(m)$ are each bounded operators on $V$, and thus so is their composition $a(n)b(m)$. For (2), notice that since there is a constant $M$ such that $\abs{a(n)c'} \leq M\abs{c'}$ for all $c' \in V$, we have
\[
  \abs{a(n)b(m)c} \leq M\abs{b(m)c} \to 0
\]
as $m \to \infty$, where convergence is independent of $n$. If instead $n$ grows, then $a(n)b \to 0$ and continuity of $Y(\bullet,z)$ yields $Y(a(n)b,z) \to 0$. Then uniformity of the sup-norm likewise yields $\lim_{n\to\infty} \abs{a(n)b(m)c}=0$ uniformly in $m$. Thus for every $\veps > 0$, there are at most finitely many pairs $(n,m)$ of integers $n,m\geq 0$ with $\abs{a(n)b(m)c} \geq \veps$. This establishes Property (2) above.

In order to give an equivalent formulation of $p$-adic locality, we introduce the formal $\delta$-function following Kac \cite{Kac}:
\[
  \delta(x-y) \df x^{-1}\sum_{n\in\ZZ} \left(\frac{x}{y}\right)^n.
\]
Likewise, let $\partial_x = d/dx$ and define $\partial^{(j)}_x = \frac{1}{j!}\partial^j$ for $j\geq 0$. The following result is a $p$-adic analogue of part of Theorem 2.3 of \cite{Kac}.

\begin{prop}
  \label{p:plocal}
  Let $V$ be a $p$-adic Banach space and let $Y(a,z)$, $Y(b,z) \in \cF(V)$ be $p$-adic fields on $V$. Then the following are equivalent:
  \begin{enumerate}
  \item $\lim_{n\to \infty} (x-y)^n[Y(a,x),Y(b,y)] = 0$;
  \item there exist unique series $c^j \in \pseries{\Endo(V)}{y,y^{-1}}$ with $\lim_{j \to \infty}c^j = 0$ such that
    \[
  [Y(a,x),Y(b,y)] = \sum_{j=0}^\infty c^j \partial_y^{(j)}\delta(x-y).
    \]
  \end{enumerate}
\end{prop}
\begin{proof}
  First, notice that by part (c) of Proposition 2.2 of \cite{Kac}, we can uniquely write
 \begin{align*}
   [Y(a,x),Y(b,y)] &= \sum_{j\geq 0} c^j(y)\partial_y^{(j)}(x-y)+b(x,y),\\
   b(x,y) &= \sum_{\substack{m\in\ZZ_{\geq 0}\\ n\in \ZZ}} a_{m,n}x^my^n.
 \end{align*}
 for series $c^j(y) \in \pseries{\Endo(V)}{y,y^{-1}}$ and $a_{m,n} \in V$.

 Suppose that (1) holds. Then by part (e) of Proposition 2.1 of \cite{Kac}, we find that $\lim_{n\to\infty} (x-y)^nb(x,y)=0$. Since $b(x,y)$ is constant and the sequence $(x-y)^n$ does \emph{not} have a $p$-adic limit, the only way this can transpire is if $b(x,y)=0$. Then by parts (e) and (d)  of Proposition 2.1 in \cite{Kac}, we now have
 \begin{equation}
   \label{eq:locality}
   (x-y)^n[Y(a,x),Y(b,y)] =\sum_{j=0}^\infty c^{j+n}(y)\partial_y^{(j)}(x-y)
 \end{equation}
 By equation (2.2.2) of \cite{Kac}, the coefficient of the $x^{-1}$-term in this expression is the series $c^{j+n}(y)$. By definition of the sup-norm on formal series, as $n$ grows, then since $(x-y)^n[Y(a,x),Y(b,y)]$ tends to zero $p$-adically, we find that $c^{j+n}(y)$ must have coefficients that become more and more highly divisible by $p$. Thus $\lim_{j} c^j=0$ in the sup-norm, which confirms that (2) holds.

 Conversely, if (2) holds, then we deduce that equation \eqref{eq:locality} holds as in the previous part of this proof. Since the coefficients of $\delta^{(j)}_y(x-y)$ are integers, and the coefficients of $c^{n}(y)$ go to zero $p$-adically in the sup-norm as $n$ grows, we find that (1) follows from equation \eqref{eq:locality} by the strong triangle ineqaulity.
\end{proof}

Recall formula (2.1.5b) of \cite{Kac}:
\[
  \partial^{(j)}_y\delta(x-y) = \sum_{m\in\ZZ} \binom{m}{j}x^{-m-1}y^{m-j}.
\]
Thus, if $Y(a,x)$ and $Y(b,y)$ are mutually $p$-adically local, then we deduce that
\begin{align*}
  [Y(a,x),Y(b,y)] =& \sum_{j=0}^\infty  c^j(y)\partial^{(j)}_y(x-y) \\
  =& \sum_{j=0}^\infty\sum_{m\in\ZZ} \binom{m}{j}c^j(y)x^{-m-1}y^{m-j}\\
  =&\sum_{j=0}^\infty\sum_{m\in\ZZ}\sum_{n\in\ZZ} \binom{m}{j}c^j(n)x^{-m-1}y^{m-n-j-1}\\
  =&\sum_{m\in\ZZ}\sum_{n\in\ZZ} \left(\sum_{j=0}^\infty\binom{m}{j}c^j(m+n-j)\right)x^{-m-1}y^{-n-1}
\end{align*}
Thus, we deduce the fundamental identity:
\begin{equation}
  \label{eq:commutator1}
  [a(m),b(n)] = \sum_{j=0}^\infty \binom{m}{j} c^j(m+n-j).
\end{equation}
This series converges by the strong triangle inequality since $c^j \to 0$.
\begin{rmk}
  As a consequence of this identity, one can deduce the $p$-adic operator product expansion as in equation (2.3.7b) of \cite{Kac}. The only difference is that it now involves an infinite sum: if $Y(a,x)$ and $Y(b,y)$ are two mutually local $p$-adic fields, then the coefficients of the $c^j$ in equation \eqref{eq:commutator1} leads to an expression
  \[
    Y(a,x)Y(b,y) \sim \sum_{j=0}^\infty \frac{c^j(y)}{(x-y)^{j+1}}
  \]
  which is strictly an abuse of notation that must be interpreted as in \cite{Kac}. Again, this $p$-adic OPE converges thanks to the fact that $c^j \to 0$.
\end{rmk}

As in the algebraic theory of VOAs, when $Y(a,x)$ and $Y(b,x)$ are mutually $p$-adically local, one has the identity $c^j(n) = (a(j)b)(n)$ and equation \eqref{eq:commutator1} specializes to the commutator formula \eqref{commform}. 
\begin{dfn}
Let $V$ be a $p$-adic vertex algebra, and let $L(V) \subseteq \Endo(V)$ denote the $p$-adic closure of the linear span of the modes of every field $Y(a,x)$ for $a \in V$.
\end{dfn}

\begin{thm}
The space $L(V)$ is a Lie algebra with commutators acting as Lie bracket.
\end{thm}
\begin{proof}
  It is clear that $L(V)$ is a subspace, so it remains to prove that it is closed under commutators. Let $L\subseteq L(V)$ denote the dense subspace spanned by the modes. First observe that if $a,b \in V$, then $[a(n),b(m)] \in L(V)$ for all $n,m\in\ZZ$, thanks to equation \eqref{commform}. It follows by linearity that $[L,L]\subseteq L(V)$. The closure of $[L,L]$ is equal to $[L(V),L(V)]$, and so we deduce $[L(V),L(V)]\subseteq L(V)$, which concludes the proof.
\end{proof}

Finally, we elucidate some aspects of rationality of locality in this $p$-adic context, modeling our discussion on Proposition 3.2.7 of \cite{LepowskyLi}. Let $V$ be a $p$-adic vertex algebra with a completed graded decomposition
\[
  V = \widehat{\bigoplus}_{n} V_{(n)},
\]
where $V_n=0$ if $n \ll 0$. The direct sum $V'$ of the duals of each finite-dimensional subspace $V_{(n)}$ consists of the linear functionals $\ell$ on $V$ such that $\ell|_{V_{(n)}} = 0$ for all but finitely many $n\in \ZZ$. This space is not $p$-adically complete, so we let $V^*$ denote its completion.
\begin{lem}
The space $V^*$ is the full continuous linear dual of $V$.
\end{lem}
\begin{proof}
  Let $\ell \colon V \to  \Qp$ be a continuous linear functional. We must show that $\ell$ is approximated arbitrarily well by functionals in the restricted dual $V'$. Continuity of $\ell$ asserts the existence of $M$ such that $\abs{\ell(v)} \leq M\abs{v}$ for all $v \in V$. Let $\delta_n \in V'$ denote the functional that is the identity on $V_{(n)}$ and zero on $V_{(m)}$ for $m\neq n$, and write
  \[
  \ell_N = \sum_{-N\leq n\leq N} \delta_n\ell.
\]
We clearly have $\ell_N \in V'$ for all $N$. Then $\ell_N \to \ell$ expresses $\ell$ as a limit of elements of $V'$. This concludes the proof.
\end{proof}

Note that we did not need the condition $V_{(n)}=0$ if $n$ is small enough in the preceding proof.

\begin{dfn}
The \emph{Tate algebra} $\Qp\langle x\rangle$ consists of series in $\pseries{\Qp}{x}$ whose coefficients go to zero.
\end{dfn}

The Tate algebra is the ring of rigid analytic functions on the closed unit disc in $\Qp$ that are defined over $\Qp$. More general Tate algebras are the building blocks of rigid analytic geometry in the same way that polynomial algebras are the building blocks of algebraic geometry. Below we use the slight abuse of notation $\Qp\langle x,x^{-1}\rangle$ to denote formal series in $x$ and $x^{-1}$ whose coefficients go to zero in both directions. Elements of $\Qp\langle x^{-1}\rangle[x]$ can be interpreted as functions on the region $\abs{x} \geq 1$, while elements of $\Qp\langle x,x^{-1}\rangle$ can be interpreted as functions on the region $\abs{x} =1$. Unlike in complex analysis, this boundary circle defined by $\abs{x}=1$ is a perfectly good rigid-analytic space.

\begin{lem}
  \label{l:rationality1}
  Let $V$ be a $p$-adic VOA, let $u,v \in V$, let $\ell_1 \in V'$, and let $\ell_2 \in V^*$. Then
  \begin{align*}
    \langle \ell_1,Y(u,x)v\rangle &\in \Qp \langle x^{-1}\rangle[x],\\
    \langle \ell_2,Y(u,x)v\rangle &\in \Qp\langle x,x^{-1}\rangle.
  \end{align*}
\end{lem}
\begin{proof}
Write $u = \sum_{n\in \ZZ} u_n$ and $v = \sum_{n \in \ZZ} v_n$ where $u_n,v_n \in V_{(n)}$ for all $n$, and $\lim_{n\to \infty} u_n = 0$, $\lim_{n \to \infty} v_n = 0$. Since each $u(n)$ is continuous, we have
\[
  u(n)v = \sum_{b\in \ZZ} u(n)v_b.
\]
Likewise, since $u \mapsto Y(u,x)$ is continuous, we have $Y(u,x)= \sum_a Y(u_a,x)$ and hence
\[
  u(n)v = \sum_{a,b\in\ZZ} u_a(n)v_b.
\]
Since $u_a$ and $v_b$ are homogeneous of weight $a$ and $b$, respectively, we have that $u_a(n)v_b \in V_{a+b-n-1}$. If $\ell \colon V \to \Qp$ is a continuous linear functional, and if $V_{(n)} = 0$ for $n < M$, then we have
\[
  \ell(u(n)v) = \sum_{\substack{a,b\geq M\\ a+b\geq M+n+1}} \ell(u_a(n)v_b).
\]

First suppose that $\ell \in V'$, so that $\ell|_{V_{(n)}} = 0$ if $n > N$, where $M \leq N$. Hence in this case we have

\[
  \ell(u(n)v) = \sum_{a\geq M}\sum_{b=M+n+1-a}^{N+n+1-a}\ell(u_a(n)v_b)
\]
Suppose that $N+n+1-a < M$. Then each $v_b$ in the sum above must vanish, so that we can write the sum in fact as
\[
  \ell(u(n)v) = \sum_{a=M}^{N-M+n+1}\sum_{b=M+n+1-a}^{N+n+1-a}\ell(u_a(n)v_b)
\]
If $M > N-M+n+1$, or equivalently, $2M-N-1 > n$, then the sum on $a$ is empty and $\ell(u(n)v)$ vanishes. Therefore, $\langle \ell,Y(u,x)v\rangle$ has only finitely many nonzero terms in positive powers of $x$. If $n\to \infty$ then since $u(n)v\to 0$, and $\ell$ is continuous, we likewise see that the coefficients of the $x^{-n-1}$ terms go to zero as $n$ tends to infinity. This establishes the first claim of the lemma.

Suppose instead that $\ell \in V^*$, let $C$ be a constant such that $\abs{u(n)v} \leq C\abs{v}$ for all $n$ (we may assume without loss that $v\neq 0$), and write $\ell = \ell_1+\ell'$ where $\ell_1 \in V'$ and $\abs{\ell'} < \frac{\veps}{C\abs{v}}$. Then for each $n$,
\[
\abs{\ell'(u(n)v)} \leq \frac{\veps}{C\abs{v}}\abs{u(n)v} \leq \veps
\]
and so
\[
  \abs{\ell(u(n)v)} \leq \sup(\abs{\ell_1(u(n)v)}, \abs{\ell'(u(n)v)})\leq\sup(\abs{\ell_1(u(n)v)}, \veps)
\]
By the previous part of this proof, as $n$ tends to $-\infty$, eventually $\ell_1(u(n)v)$ vanishes. We thus see that $\ell(u(n)v) \to 0$ as $n \to -\infty$. Since $\ell$ is a continuous linear functional and $u(n)v \to 0$ as $n \to \infty$, we likewise get that $\ell(u(n)v) \to 0$ as $n \to \infty$. This concludes the proof.
\end{proof}

\begin{rmk}
A full discussion of rationality would involve a study of series
\[
  \langle \ell, Y(a,x)Y(b,y)v\rangle
\]
for $v \in V$ and $\ell \in V^*$, where $a$ and $b$ are mutually $p$-adically local. Ideally, one would show that such series live in some of the standard rings arising in $p$-adic geometry, similar to the situation of Lemma \ref{l:rationality1} above. It may be necessary to impose stronger conditions on rates of convergence of limits such as $\lim_{n\to \infty} a(n)b=0$ in the definition of $p$-adic field in order to achieve such results. We do not pursue this study here.
\end{rmk}

\section{The Heisenberg algebra}
\label{s:heisenberg}

We now discuss $p$-adic completions of the Heisenberg VOA of rank $1$, i.e., $c=1$ in detail, which was the motivation for some of the preceding discussion. In general there are many possible ways to complete such VOAs, as illustrated below, though we only endow a sup-norm completion with a $p$-adic VOA structure. Our notation follows \cite{Lepowsky}. To begin, define
\[
  S = \Qp[h_{-1},h_{-2},\ldots],
\]
a polynomial ring in infinitely many indeterminates. This ring carries a natural action of the Heisenberg Lie algebra. Recall that this algebra is defined by generators $h_n$ for $n\in \ZZ \setminus \{0\}$ and the central element $1$, subject to the canonical commutation relations
\begin{eqnarray}\label{CCRs}
  [h_m,h_n]=m\delta_{m+n,0}1. 
\end{eqnarray}
The Heisenberg algebra acts on $S$ as follows: if $n < 0$ then $h_n$ acts by multiplication, $1$ acts as the identity, while for $n > 0$, the generator $h_n$ acts as $n\frac{d}{dh_{-n}}$. Then $S$ carries a natural structure of VOA as discussed in Chapter 2 of \cite{FBZ}, where the Virasoro action is given by
\begin{align*}
  c &\mapsto 1,\\
  L_n &\mapsto \frac 12 \sum_{j\in\ZZ} h_jh_{n-j}, \quad n\neq 0,\\
  L_0 & \mapsto \frac 12 \sum_{j\in\ZZ} h_{-\abs{j}}h_{\abs{j}}.
\end{align*}

Elements of $S$ are polynomials with finitely many terms, and this space can be endowed with a variety of $p$-adic norms. We describe a family of such norms that are indexed by a real parameter $r\geq 0$, which are analogues of norms discussed in Chapter 8 of \cite{Kedlaya} for polynomial rings in infinitely many variables.

For $I$ a finite multi-subset of $\ZZ_{< 0}$, let $h^I = \prod_{i \in I} h_i$ and define $\abs{I} = -\sum_{i \in I}i$, so that $h^I$ has degree $\abs{I}$. For fixed $r \in \RR_{>0}$ define a norm
\[
\abs{\sum_{I} a_I h^I}_r = \sup_I\abs{a_I}r^{\abs{I}}.
\]
For example, when $r = 1$, this norm agrees with the sup-norm corresponding to the integral basis given by the monomials $h^I$. Let $S_r$ denote the completion of $S$ relative to the norm $a \mapsto \abs{a}_r$.
\begin{prop}
  The ring $S_r$ consists of all series $\sum_{I}a_Ih^I \in \pseries{\Qp}{h(-1),h(-2),\ldots}$ such that
  \[
\lim_{\abs{I} \to \infty} \abs{a_I}r^{\abs I} = 0.
  \]
\end{prop}
\begin{proof}
  Let $a^n$ denote a Cauchy sequence in $S$ relative to the norm $\abs{a}_r$, so that we can write $a^n = \sum_{I}a^n_Ih^I$ for each $n\geq 0$. We first show that for each fixed indexing multiset $J$, the sequence $a_J^n$ is Cauchy in $\Qp$ and thus has a well-defined limit. For this, let $\veps > 0$ be given and choose $N$ such that $\abs{a^n-a^m} < \veps r^{\abs{J}}$ for all $n,m > N$.  This means that
  \[
  \abs{a^n_J-a^m_J} r^{\abs{J}} \leq \abs{a^n-a^m} < \veps r^{\abs{J}}
\]
for all $n,m > N$, so that the sequence $(a^n_J)$ is indeed Cauchy. Let $a_J$ denote its limit in $\Qp$ and let $a = \sum_J a_Jh^J$.

There exists an index $n$ such that $\abs{a-a^n} < \veps$. But notice that since $a^n \in S$, it follows that $a^n_I = 0$ save for finitely many multisets $I$. In particular, if $\abs{I}$ is large enough we see that
\[
  \abs{a_I}r^{\abs{I}} = \abs{a_I-a_I^n}r^{\abs{I}} \leq \abs{a-a^n} < \veps.
\]
Therefore $\lim_{\abs{I} \to \infty} \abs{a_I}r^{\abs{I}}=0$ as claimed, and this concludes the proof.
\end{proof}

\begin{cor}
If $0 < r_1 < r_2$ then there is a natural inclusion $S_{r_2} \subseteq S_{r_1}$.
\end{cor}
\begin{proof}
This follows immediately from the previous proposition since $r_1 < r_2$ implies that $\abs{a_I}r_1^{\abs{I}} < \abs{a_I} r_2^{\abs{I}}$.
\end{proof}

\begin{rmk}
  Let $\widehat S = \bigcap_{r > 0} S_r$. The ring $\widehat S$ consists of all series $\sum_{I}a_Ih^I$ such that $\lim_{\abs{I}\to\infty} \abs{a_I}r^{\abs{I}}=0$ for all $r > 0$. It contains $S$ but is strictly larger: for example, the infinite series $\sum_{n\geq 0}p^{n^2}h_{-1}^n$ is contained in $\widehat S$ but, being an infinite series, it is not in $S$. This ring is an example of a $p$-adic Frechet space that is not a Banach space. Therefore, according to our definitions, $\widehat{S}$ does not have a structure of $p$-adic VOA. It may be desirable to extend the definitions to incorporate examples like this into the theory. 
\end{rmk}

\begin{lem}
  \label{l:hsubmult}
If $a,b \in S_r$, then $\abs{a(n)b}_r\leq \abs{a}_r\abs{b}_rr^{-n-1}$ for all $n\in\ZZ$.
\end{lem}
\begin{proof}
Write $a = \sum_{I} a_Ih^I$ and $b = \sum_I b_Ih^I$ where $\abs{a_I}r^{\abs{I}} \to 0$ and $\abs{b_I}r^{\abs{I}} \to 0$. Then the ultrametric property gives
  \[
  \abs{a(n)b}_r \leq \sup_{I,J} \abs{a_Ib_J}\abs{h^I(n)h^J}_r.
\]
Notice that $h^I$ is homogeneous of weight $\abs{I}$, and $h^{J}$ is homogeneous of weight $\abs{J}$. Therefore, it follows that $h^I(n)h^J$ is homogeneous of weight $\abs{I}+\abs{J}-n-1$, and thus $\abs{h^I(n)h^J}_r \leq r^{\abs{I}+\abs{J}-n-1}$. Combining this observation with the lined inequality above establishes the lemma.
\end{proof}

The preceding lemma illustrates that it is nontrivial to obtain uniform bounds for $\abs{a(n)b}_r/\abs{a}_r$ unless $r=1$. In this case, however, we obtain:
\begin{prop}
The Banach ring $S_1$ has a natural structure of submultiplicative $p$-adic VOA.
\end{prop}
\begin{proof}
Since $S$ has an integral basis and $\abs{\cdot}_1$ is the corresponding sup-norm, we may apply Proposition \ref{p:integralcompletion} to conclude that $S_1$ has the structure of a $p$-adic VOA. That it is submultiplicative follows from Lemma \ref{l:hsubmult}.
\end{proof}

In \cite{DMN}, the authors show that there is a surjective \emph{character map}
\begin{eqnarray}\label{qmodchars}
S \to \Qp[E_2,E_4,E_6]\eta^{-1} 
\end{eqnarray}
of the Heisenberg algebra $S$ onto the  free-module of rank $1$ generated by $\eta^{-1}$ over the ring of quasi-modular forms of level one with $p$-adic coefficients\footnote{Actually, the authors work over $\CC$, but their proof applies to any field of characteristic zero.}. After adjusting the grading on the Fock space $S$, this is even a map of graded modules. See  \cite{MT2, MT} for more details on this map.

The character is defined as follows: if $v \in S$ is homogeneous of degree $k$, then $v(n)$ is a graded map that increases degrees by $k-n-1$. In particular, $v(k-1)$ preserves the grading, and we write $o(v) = v(k-1)$ for homogeneous $v$ of weight $k$. This is the \emph{zero-mode} of $v$. The zero mode is then extended to all of $S$ by linearity. With this notation, let $S^{(n)}$ be the $n$th graded piece of $S$ and define the \emph{character} of $v \in S$ by the formula
\[
  Z(v,q) \df q^{-1/24}\sum_{n\geq 0} \Tr_{S^{(n)}}(o(v))q^n.
\]
\begin{thm}
  \label{t:MT}
  The association $v \mapsto \eta\cdot Z(v,q)$ defines a surjective $\Qp$-linear map
  \[
  f\colon S \to \Qp[E_2,E_4,E_6],
\]
where $\eta$ is the Dedekind $\eta$-function. $\hfill\Box$
\end{thm}

Our goal now is to use $p$-adic continuity to promote this to a map from $S_1$ into Serre's ring $M_p$ of $p$-adic modular forms as defined in \cite{Serre2}. Recall that $M_p$ is the completion of $\Qp[E_4,E_6]$ with respect to the $p$-adic sup-norm on $q$-expansions. When $p$ is odd, Serre proved that $\Qp[E_2,E_4,E_6]\subseteq M_p$. 
\begin{thm}
  \label{t:heisenbergcharacters}
  Let $p$ be an odd prime. Then the map $v \mapsto \eta Z(v,q)$ on the Heisenberg algebra $S$ extends to a natural $\Qp$-linear map
  \[
  f\colon S_1\to M_p.
\]
The image contains all quasi-modular forms of level one.
\end{thm}
\begin{proof}
  Recall that $S_1$ is the completion of $S$ with respect to the $p$-adic sup-norm. Since the rescaling factor of $\eta$ will not affect the continuity of $f$, we see that to establish the continuity of $f$, we are reduced to proving the $p$-adic continuity of the map $v \mapsto Z(v,q)$, where the image space $q^{-1/24}\Qp\langle q\rangle$ is given the $p$-adic sup-norm. Suppose that $\abs{u-v}_1 < 1/p^k$, so that we can write $u=v+p^kw$ for some $w \in S_1$ with $\abs{w}_1 \leq 1$. This means that $w$ is contained in the completion of $\Zp[h(-1),h(-2),\ldots]$, and all of its modes are defined over $\Zp$ and satisfy $\abs{w(n)}_1 \leq 1$. In particular, the zero mode $o(w)$ is defined over $\Zp$, and thus so is its trace, so that $\abs{\Tr_{S^{(n)}}(o(w))} \leq 1$ for all $n$.

 By linearity of the zero-mode map $o$ and the trace, we find that with respect to the sup-norm on $q^{-1/24}\Qp\langle q\rangle$,
\[
  \abs{Z(u,q)-Z(v,q)} = \abs{q^{-1/24}\sum_{n\geq 0} p^k\Tr_{S^{(n)}}(o(w))q^n} < p^{-k}.
\]
This shows that the character map preserves $p$-adic limits. It then follows by general topology (e.g., Theorem 21.3 of \cite{Munkres}) that since $S_1$ is a metric space, $f$ is continuous.
\end{proof}

The map in Theorem \ref{t:heisenbergcharacters} has an enormous kernel, as even the algebraic map that it extends has an enormous kernel. This complicates somewhat the study of the image, as there do not exist canonical lifts of modular forms to states in the Heisenberg VOA. A natural question is whether states in $S_1$ map onto nonclassical specializations (that is, specializations of non-integral weight) of the Eisenstein family discussed in \cite{Serre2} and elsewhere. In the next section we give some indication that this can be done, at least for certain $p$-adic modular forms, in spite of the large kernel of the map $f$.

\section{Kummer congruences in the \texorpdfstring{$p$}{p}-adic Heisenberg VOA}
\label{s:kummer}
We use the notation from Section \ref{s:heisenberg}, in particular $S_1$ is the $p$-adic Heisenberg VOA associated to the rank $1$ algebraic Heisenberg
VOA $S$ with canonical weight $1$ state $h$ satisfying the canonical commutator relations \eqref{CCRs}. We shall write down some explicit algebraic states in $S$ that converge $p$-adically in $S_1$ and we shall describe their images under the character map $f$ of Theorem \ref{t:heisenbergcharacters} which, as we have seen, is $p$-adically continuous. Thus, the $f$-image of our $p$-adic limit states will be Serre $p$-adic modular forms in $M_p$. The description of the states relies on the square-bracket formalism of \cite{Zhu}; see \cite{MT2} for a detailed discussion of this material. The convergence of the states relies, perhaps unsurprisingly, on the classical Kummer congruences for Bernoulli numbers.

In order to describe the states in $S$ of interest to us we must review some details from the theory of algebraic VOAs, indeed the part that leads to the proof of Theorem \ref{t:MT}. This involves the square bracket vertex operators and states. For a succinct overview of this we refer the reader to Subsection 2.7 of \cite{MT2}. From this we really only need the definition of the new operators $h[n]$ acting on $S$, which is as follows:
\begin{eqnarray}\label{Ysquare}
Y[h, z]:=\sum_{n\in \ZZ} h[n] z^{-n-1}:=e^zY(h, e^z-1).
\end{eqnarray}
This reexpresses the vertex operators as objects living on a torus rather than a sphere, a geometric perspective that is well-explained in Chapter 5 of \cite{FBZ}. The following result is a special case of a result proved in \cite{MT} and exposed in \cite[Theorem 4.5 and equation (44)]{MT2}.
\begin{thm}
  For a  positive odd integer $r$ we have
\[
\eta Z(h[-r]h[-1]\bone, q) = \frac{2}{(r-1)!}G_{r+1}(\tau)
\]
Here, $G_k(\tau)$ is the weight $k$ Eisenstein series
\[
G_k(\tau)\df -\frac{B_k}{2k}+\sum_{n\geq 1} \sigma_{k-1}(n)q^n
\]
where $B_k$ is the $k^{th}$ Bernoulli number.   $\hfill\Box$
\end{thm}

In order to assess convergence, we rewrite the square bracket state $h[-r]h[-1]\bone $ in terms of the  basis $\{h(-n_1)...h(-n_s)\mathbf{1}\}\ (n_1\geq n_2\geq ... \geq n_s\geq 1)$ of $S$.
\begin{lem}
  \label{l:squaretoround}
We have 
\[
(r-1)! h[-r]h[-1]\bone = \sum_{m =0}^{r+1}  c(r, m) h(-m-1)h(-1)\mathbf{1}-\frac{B_{r+1}}{r+1}\mathbf{1},
\]
where
\[
  c(r, m)\df  \sum_{j=0}^m (-1)^{m+j}{m \choose j}(j+1)^{r-1}= m!\stirling{r}{m+1}
\]
and $\stirling{r}{m+1}$ denotes a Stirling number of the second kind. In particular, $c(r,m) = 0$ if $m\geq r$.
\end{lem}
\begin{proof}
  It is readily found from \eqref{Ysquare} that  $h[-1]\bone=h(-1)\bone=h$.\ Then we calculate: 
\begin{align*}
  &(r-1)! h[-r]h[-1]\bone \\
=& (r-1)! \Res_z z^{-r} Y[h, z]h(-1)\bone\\
=&(r-1)! \Res_z z^{-r}e^z Y(h, e^z-1)h(-1)\bone\\
=&(r-1)! \Res_z z^{-r}e^z \left\{\sum_{m\geq 0} h(-m-1)h(-1)\bone (e^z-1)^m+\bone (e^z-1)^{-2} \right\}.
\end{align*}
This, then,  is the desired expression for $(r-1)! h[-r]h[-1]\bone$ in terms of our preferred basis for $S$, and it only remains to sort out the numerical coefficients.\ Taking the second summand in the previous display first, we find by a standard calculation that
\begin{align*}
\Res_z z^{-r} (r-1)! e^z(e^z-1)^{-2}=& - (r-1)! \Res_z z^{-r}\frac{d}{dz}\left( (e^z-1)^{-1}  \right)\\
=&  - (r-1)! \Res_z z^{-r}\frac{d}{dz}\left( \sum_{\ell\geq 0} \frac{B_{\ell}}{\ell !}z^{\ell-1}  \right)\\
=&  - (r-1)! r\frac{B_{r+1}}{(r+1)!} =  - \frac{B_{r+1}}{r+1},
\end{align*}
thereby confirming the coefficient of $\bone$ as stated in the Lemma.\ As for the first summand, for each $0\leq m\leq r+1$ the needed coefficient $c(r,m)$ is equal to
\begin{align*}
(r-1)! \Res_z z^{-r}e^z (e^z-1)^m
=&(r-1)! \Res_z z^{-r} \sum_{j=0}^m (-1)^{m+j}{m\choose j}e^{(j+1)z}\\
=&(r-1)! \Res_z z^{-r} \sum_{j=0}^m (-1)^{m+j}{m\choose j}\sum_{n=0}^{\infty} \frac{(j+1)^n}{n!}z^n\\
=&\sum_{j=0}^m (-1)^{m+j}{m\choose j} (j+1)^{r-1}
\end{align*}
as required. The relationship between $c(r,m)$ and  Stirling numbers of the second kind then follows by standard results, and this completes the proof of the Lemma.
\end{proof}

Combining the last two results, we obtain
\begin{cor}
  For a positive odd integer $r$, define states $v_r$ in the rank $1$ algebraic Heisenberg VOA $S$ as follows:
\begin{eqnarray*}
v_{r}:=1/2(r-1)!h[-r]h[-1]\mathbf{1}.
\end{eqnarray*}
Then
\begin{eqnarray*}
v_{r} = 1/2 \left\{\sum_{m =0}^{\infty}  c(r, m) h(-m-1)h(-1)\mathbf{1}  -\frac{B_{r+1}}{r+1}\mathbf{1} \right\}
\end{eqnarray*}
and
\begin{eqnarray*}
f(v_{r}) = G_{r+1}.
\end{eqnarray*}
$\hfill\Box$
\end{cor}

We now wish to study the $p$-adic properties of these states as $r$ varies.\ We shall rescale things in a convenient way:
\[
  u_{r} \df (1-p^{r})\left(\sum_{m =0}^{\infty}  c(r, m) h(-m-1)h(-1)\mathbf{1}  -\frac{B_{r+1}}{r+1}\mathbf{1}\right) = 2(1-p^r)v_r.
\]
We now lift the classical Kummer congruences for Bernoulli numbers to this sequence of states in the $p$-adic Heisenberg VOA. We begin with a simple Lemma treating the $c(r,m)$ terms.

\begin{lem}
  Let $p$ be an odd prime with $r=1+p^a(p-1), s=1+p^b(p-1)$ and  $a\leq b$. Then for all $m$ we have $c(r, m)\equiv c(s, m)$ (mod $p^{a+1}$).
\end{lem}
\begin{proof} Recall that 
\[
c(r, m)\df  \sum_{j=0}^m (-1)^{m+j}{m \choose j}(j+1)^{r-1}.
\]
If $p\mid j+1$ then certainly $(j+1)^{r-1}\equiv 0 \pmod{p^{a+1}}$. On the other hand, if $p\nmid j+1$ then $(j+1)^{r-1}\equiv 1 \pmod{p^{a+1}}$. It follows that
\[
c(r, m) \equiv  \sum_{\substack{j=0\\  p\nmid j+1}}^m (-1)^{m+j}\binom{m}{j} \pmod{p^{a+1}},
\]
and since the right side of the congruence is independent of $r$ we deduce that $c(r,m) \equiv c(s,m)\pmod{p^{a+1}}$. This proves the Lemma.
\end{proof}

\begin{thm}[Kummer congruences]
The sequence $(u_{1+p^a(p-1)})_{a\geq 0}$ converges $p$-adically in $S_1$ to a state that we denote $u'_{1} := \lim_{a\to \infty} u_{1+p^a(p-1)}$.
\end{thm}
\begin{proof}
 The convergence for the terms in this sequence involving the Bernoulli numbers follows from the classical Kummer congruences. Therefore, if $a\leq b$, it will suffice to establish that,
  \[
  (1-p^{1+p^a(p-1)})c(1+p^a(p-1),m) \equiv (1-p^{1+p^b(p-1)})c(1+p^b(p-1),m) \pmod{p^{a+1}}
\]
for all $m$ in the range $0\leq m \leq p^b(p-1)$. But this follows by the preceding Lemma.
\end{proof}

Notice that, putting all of this together, we have
\begin{align*}
  f(u'_{1}) &= \lim_{a\to \infty}(1-p^{1+p^a(p-1)})2G_{2+p^a(p-1)}\\
               &=2\lim_{a\to \infty}(1-p^{1+p^a(p-1)})G_{2+p^a(p-1)}\\
  &= 2G_2^*,
\end{align*}
where $G_2^*$ is the $p$-normalized Eisenstein series encountered in \cite{Serre2}, with $q$-expansion given by
\[
G_2^*(q) = \frac{p^2-1}{24}+\sum_{n\geq 1} \sigma^*(n)q^n.
\]
Here $\sigma^*(n)$ is the sum over the divisors of $n$ that are coprime to $p$.\

This computation illustrates the fact that the $p$-adic Heisenberg algebra contains states that map under $f$ to  $p$-adic modular forms that are \textit{not} quasi-modular forms of level $1$.\ It seems plausible that the rescaled $p$-adic character map $f$ of Theorem \ref{t:heisenbergcharacters} is \emph{surjective}, that is, for an odd prime $p$ every Serre $p$-adic modular form can be obtained from the normalized character map applied to a sequence of $p$-adically converging states in $S_1$. Aside from some other computations with Eisenstein series, we have yet to examine this possibility in detail.

\section{The \texorpdfstring{$p$}{p}-adic Moonshine module}
\label{s:monster}
The Moonshine module $V^{\natural}=(V^{\natural}, Y, \bone, \omega)$ is a widely known example of an algebraic VOA over $\CC$ \cite{FLM}.\ Its notoriety is due mainly to the fact that its automorphism group is the Monster simple group $F_1$ (loc.\ cit.) (Use of $F_1$ here follows the Atlas notation \cite{Atlas} rather than the more  conventional $M$ in order to avoid confusion with the Heisenberg VOA).  The conformal grading on $V^{\natural}$ takes the general shape
\begin{eqnarray}\label{grading}
V^{\natural} =\CC\bone \oplus V^{\natural}_{(2)} \oplus \cdots
\end{eqnarray}
and we denote by $ V^{\natural}_{+}\df V^{\natural}_{(2)} \oplus \cdots$ the summand consisting of the homogeneous vectors of positive weight.

Dong and Griess proved \cite[Theorem 5.9]{DongGriess1} that $V^{\natural}$ has an $F_1$-invariant integral form $R$ over $\ZZ$ as in Definition \ref{d:lattice} (see also \cite{Carnahan}) and in particular $V^{\natural}$ is susceptible to the kind of analysis we have been considering, thereby giving rise to the $p$-adic Moonshine module $V_p^{\natural}$. In this Section we give some details.

We obtain an algebraic vertex algebra over $\ZZ_p$ by extension of scalars $U\df R\otimes \ZZ_p$, and similarly a vertex algebra over $\QQ_p$, namely $V\df R\otimes_{\ZZ} \QQ_p= U\otimes_{\ZZ_p} \QQ_p$. Note that $\omega \in V$  thanks to property (iii) of Definition \ref{d:lattice}, so $V$ is in fact a VOA over $\QQ_p$. Now define 
\[
V_p^{\natural} \df V' = \ \textrm{the completion of $V$ with respect to its sup-norm.} 
\]
Comparison  with Proposition \ref{p:integralcompletion} shows that $V_p^{\natural}$ is a $p$-adic VOA and $U$ furnishes an integral structure according to Definition \ref{d:integral} and Proposition \ref{propint}. We can now prove the following result:
\begin{thm}
  \label{t:monsterchar}
  For each prime $p$ the $p$-adic Moonshine module $V_p^{\natural}$ has the following properties:
\begin{enumerate}
\item[(a)] $V_p^{\natural}$ has an integral structure isomorphic to the completion of the Dong-Griess form $R\otimes\ZZ_p$,
\item[(b)] The Monster simple group $F_1$ is a group of automorphisms of $V^{\natural}_p$,
\item[(c)] The character map $v\mapsto Z_{V^{\natural}}(v, q)$ for $v\in V_{+}^{\natural}$ extends to a $\QQ_p$-linear map
  \[V^{\natural}_{+}\rightarrow M_p\]
  into Serre's ring of $p$-adic modular forms. The image contains all cusp forms of integral weight and level one.
\end{enumerate}
\end{thm}
\begin{proof}
  We have already explained the proof of part (a). As for (b), notice that every submodule $p^kV^{\natural}$ is preserved by the action of $F_1$ by linearity of this action. Therefore, there is a natural well-defined action of $F_1$ on each quotient $V^\natural/p^kV^{\natural}$ module, and this action extends to the limit $V^\natural_p \df \varprojlim_{k}V^\natural/p^kV^{\natural}$. This concludes the proof of part (b).\ 
  
As for part (c), we assert that  if $v\in V^{\natural}_{+}$ then the trace function $Z_{V^{\natural}}(v, \tau)$ has a $q$-expansion of the general shape $aq+ bq^2+\cdots$. For we may assume that $v\in V^{\natural}_{(k)}$ is homogeneous of some weight $k\geq 2$. Then $o(v)\bone=v(k-1)\bone=0$ by the creation property, and since $V^{\natural}$ has central charge $c=24$ then by (\ref{grading}) we have $Z_{V^{\natural}}(v, \tau)=q^{-c/24}\sum_{n\geq 0} \Tr_{V^{\natural}_{(n)}} o(v)q^n= \Tr_{V^{\natural}_{(2)}}o(v) q+\cdots$, and our assertion is proved.

By the previous paragraph and Zhu's Theorem \cite{Zhu, DLMModular, FMCharacters}, it follows that $Z_{V^{\natural}}(v, \tau)$ is a sum of cusp-forms of level $1$ and mixed weights and by [7, Theorem 1] every such level $1$ cusp-form appears in this way.\ Now to complete the proof of part (c), proceed as in the proof of Theorem \ref{t:MT}.\ There is no need to exclude $p=2$ because $E_2$ plays no r\^{o}le.
\end{proof}

\begin{rmk}
  The manner in which $p$-adic modular forms arise from the character map for the $p$-adic VOAs in Theorems \ref{t:heisenbergcharacters} and \ref{t:monsterchar} are rather different. This difference masks a well-known fact in the theory of VOAs, namely that the algebraic Heisenberg VOA and the algebraic Moonshine module have vastly different module categories. Indeed, $V^{\natural}$ has, up to isomorphism, exactly one simple module --- namely the adjoint module $V^{\natural}$ itself --- whereas $M$ has infinitely many inequivalent simple modules. In the case of $V^{\natural}$, Zhu's theorem (loc.\ cit.)\ may be applied and it leads to the fact, explained in the course of the proof of Theorem \ref{t:monsterchar}, that trace functions for positive weight states are already sums of cusp forms of mixed weight.\ But to obtain forms without poles, we must exclude states of weight $0$ (essentially the vacuum $\bone$). On the other hand, there is no such theory for the algebraic Heisenberg and indeed one sees from \eqref{qmodchars} that the trace functions for $M$ are by no means elliptic modular forms in general. That they take the form described in \eqref{qmodchars} is a happy convenience (and certainly not generic behaviour), and we can normalize the trace functions to exclude poles by multiplying by $\eta$.\ Since the quasi-modular nature of the characters disappears when one considers the $p$-adic Heisenberg algebra, it is natural to ask whether the module category for the $p$-adic Heisenberg algebra is simpler than for the algebraic Heisenberg VOA. We do not currently have an answer for this question, indeed, we have not even given a concrete definition of the module category for a $p$-adic VOA!
\end{rmk}

\begin{rmk}
We point out that our notation $V_p^{\natural}$ for the $p$-adic Moonshine module as we have defined it may be misleading in that it does not record the dependence on the Dong-Griess form $R$.\ Indeed there are other forms that one could use in its place such as the interesting self-dual form of Carnahan \cite{Carnahan} and we have not studied whether these different forms produce isomorphic $p$-adic Moonshine modules.
\end{rmk}

Data sharing is not applicable to this article.

\medskip
The authors declare that there are no conflicts of interest.

\bibliographystyle{plain}
\bibliography{refs}
\end{document}